\documentclass[12pt]{amsart}
 \usepackage[foot]{amsaddr}
\usepackage{bbm,bm}
\usepackage{amssymb}
\usepackage[margin=1.5in, total={6.0in, 8in}]{geometry}
\usepackage{colordvi}
\usepackage{graphicx,xspace}
\usepackage{color}
\usepackage{xspace}

\usepackage[bookmarks=false,colorlinks,linkcolor=blue,citecolor=green]{hyperref}

\definecolor{mains}{cmyk}{.3, .85, .75, 0}  
\definecolor{afb}{rgb}{0.03, 0.27, 0.49}
\definecolor{def}{rgb}{0.27, 0.03, 0.49}

\newcounter{FNC}[page]
\def\fauxfootnote#1{{\addtocounter{FNC}{2}$^\fnsymbol{FNC}$%
     \let\thefootnote\relax\footnotetext{$^\fnsymbol{FNC}$\Magenta{#1}}}}

\numberwithin{equation}{section}
\newtheorem{theorem}{Theorem}[section]

\newtheorem{lemma}[theorem]{Lemma}
\newtheorem{corollary}[theorem]{Corollary}

\newtheorem{thm}[theorem]{Theorem}
\newtheorem{defn}[theorem]{Definition}
\newtheorem{exm}[theorem]{Example}
\newtheorem{rem}[theorem]{Remark}
\newenvironment{definition}[1][]{\rm\begin{defn}[#1]\rm}{\end{defn}}
\newenvironment{example}[1][]{\rm\begin{exm}[#1]\rm}{\end{exm}}

\author{Stefan Forcey} \address[S. Forcey]{
    Department of Mathematics\\
    The University of Akron\\
    Akron, OH 44325-4002
    }
    \email{sforcey@uakron.edu}  \urladdr{http://www.math.uakron.edu/\~{}sf34/}
\author{Drew Scalzo} \address[D. Scalzo]{
    Department of Mathematics\\
    The University of Akron\\
    Akron, OH 44325-4002
    }

\title[Galois Connections]{Galois connections for  phylogenetic networks  and their polytopes}

\begin{document}

\begin{abstract}
    We describe Galois connections which arise between two kinds of combinatorial structures, both of which generalize trees with labelled leaves, and then apply those connections to a family of polytopes.  
    
    The graphs we study can be imbued with metric properties or associated to vectors. Famous examples are the Billera-Holmes-Vogtmann metric space of phylogenetic trees, and the Balanced Minimal Evolution polytopes of phylogenetic trees described by Eickmeyer, Huggins, Pachter and Yoshida. Recently the space of trees has been expanded to split networks by Devadoss and Petti, while the definition of phylogenetic polytopes has been generalized to encompass 1-nested phylogenetic networks, by Durell and Forcey. The first Galois connection we describe is a reflection between the (unweighted) circular split networks and the 1-nested phylogenetic networks. Another Galois connection exists between certain metric versions of these structures. Reflection between the purely combinatorial posets becomes a coreflection in the geometric case.
    
    Our chief contributions here, beyond the discovery of the Galois connections, are: a translation between approaches using PC-trees and networks, a new way to look at weightings on networks, and a fuller characterization of faces of the phylogenetic polytopes.  
    

 \end{abstract}
\maketitle

\section{Introduction}The simplest structure we consider here is a partition of a finite set into two parts.  A tree with labeled leaves represents a collection of such bipartitions, since removing any edge of the tree partitions the leaves via separating the tree into two components. A collection of bipartitions that can be thus displayed by a tree is called \emph{pairwise compatible}.  In this paper we study two generalizations of labeled trees which display larger, more general sets of bipartitions.
 
The motivation for many of our definitions and results comes from phylogenomics. The goal is to recreate ancestral relationships using genetic data from individuals or species available today (extant taxa). The process begins with models  of gene mutation, which allow the measurement of genetic distance between taxa. Euclidean embeddings of simple graphs, called circular split networks, allow such metrics to be visualized in terms of discrete amounts of genetic distance assigned to splits between subsets of taxa. In contrast, another type of semi-labeled simple graphs, called phylogenetic networks, model the actual hereditary relationships by assigning genetic distances to directed edges. In this paper we study the interplay of those two pictures.

Figure~\ref{nuexamplo} exhibits two phylogenetic networks,  showing relationships between taxa numbered $1\dots 8$. On the left is shown an example of the undirected weighted graphs that we will discuss, and on the right is a directed network which exhibits how that graph might be interpreted in a biological setting. To achieve the directed version an \emph{outgroup} taxon is usually pre-selected; it is known to be relatively unrelated to the others. That selection allows a root node to be designated as the only source in the directed graph---it is the node at which the outgroup is attached.  

\begin{figure}
    \centering
   \includegraphics[width=\textwidth]{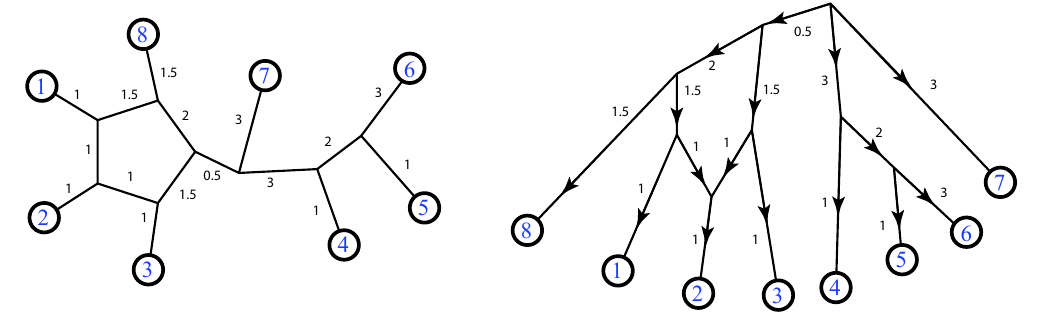}
    \caption{A weighted 1-nested phylogenetic network, unrooted on the left, and with a root node (the only source) on the right. On the right leaf 7 is the outgroup, and edges are given compatible directions. (Not all those directions are completely determined by the choice of root.) }
    \label{nuexamplo}
\end{figure}

\begin{figure}
    \centering
    \includegraphics[width=\textwidth]{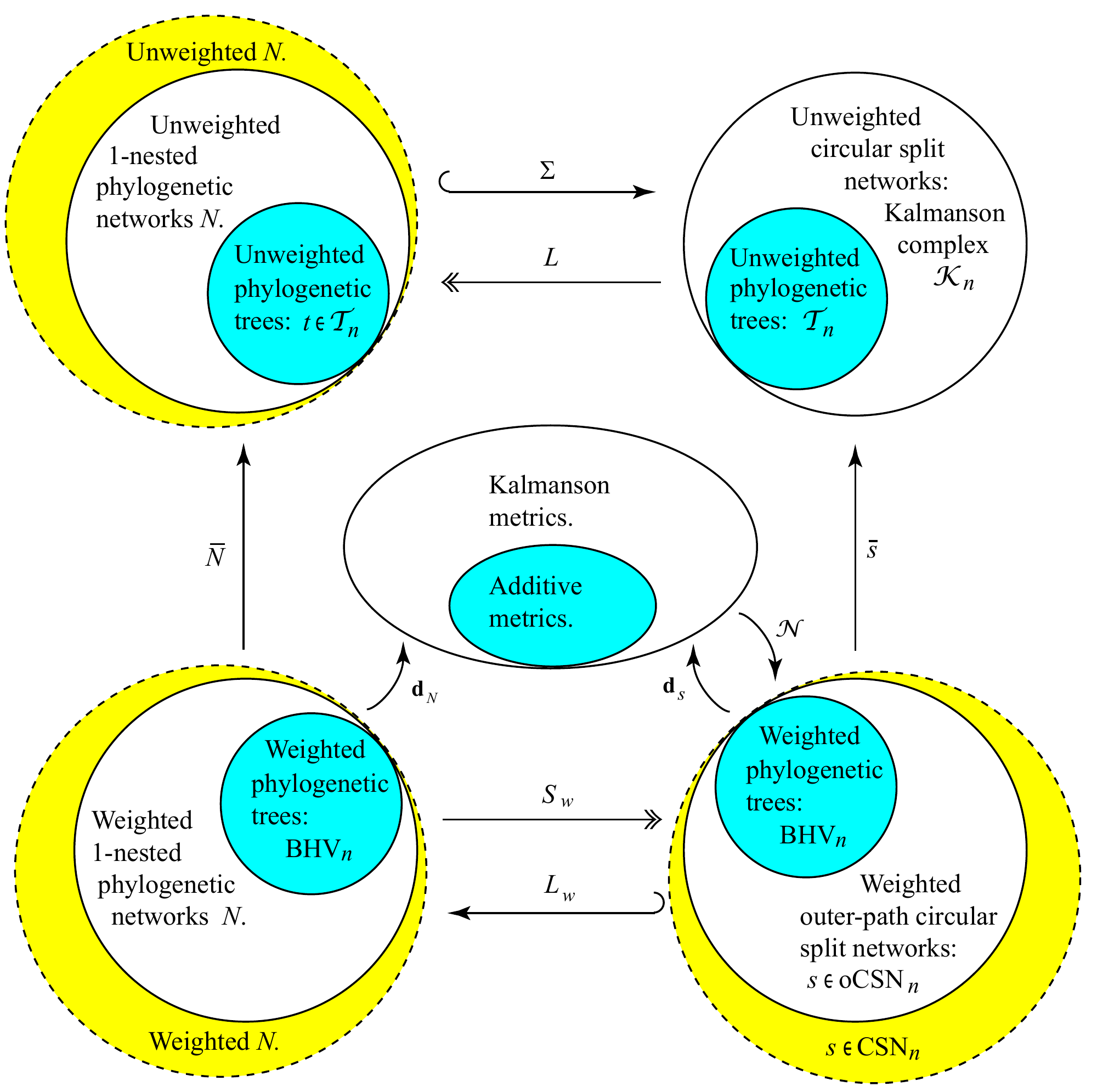}
    \caption{The sets and functions in this paper: two types of unweighted networks and their weighted versions; and the associated injections and surjections. Smaller circles are included sets:  trees are special networks. The four horizontal maps have domain and range shown as unshaded circles. When restricted to trees the four horizontal maps are identities. The vertical maps, shown via an overline, denote the forgetting of weights.}
    \label{fig:mapo}
\end{figure}

We will now leave the applications however, and focus on the underlying combinatorics. Figure~\ref{fig:mapo} shows the combinatorial structures and maps we will discuss, along with the notation used herein, for reference. The metric spaces of weighted trees and circular split networks are known as BHV${}_n$ and CSN${}_n$ respectively, studied by Billera, Holmes and Vogtmann in \cite{bhv} and by Devadoss and Petti in \cite{dev-petti}.  In Section~\ref{defs} we define the individual structures at the top of the figure: unweighted 1-nested phylogenetic networks and circular split networks. The former are studied in \cite{Gambette2017} and a simplicial complex of the latter, the Kalmanson complex $\mathcal{K}_n$, is studied in  \cite{terhorst} and \cite{dev-petti}.  In Section~\ref{order} we show how to partially order those sets. In Section~\ref{functions} we define the functions $L$ and $\Sigma$ that make up our first Galois connection, in Theorem~\ref{unwgalois}.

In Section~\ref{sec_weight} we consider the same combinatorial structures, but equipped with real-valued weighting functions. Weights are crucial to the biological applications. The second Galois connection is described between weighted phylogenetic networks and weighted split networks,  in Theorem~\ref{wgalois}.

In Section~\ref{sec_poly} we review the recent results describing the Balanced Minimum Evolution (BME) polytopes for networks,  whose faces include 1-nested networks. Then we use the Galois connections to prove, in Theorems~\ref{newth} and~\ref{subth}, that given a weighted phylogenetic network there is an associated vector whose dot product with vertices of the BME polytopes is minimized at faces corresponding to  unweighted refinements of that network.

\section{Definitions and Lemmas}\label{defs}
For convenience we work with the set $[n]=\{1,\dots,n\}.$ A \textcolor{def}{\emph{split}} $A|B$ is a bipartition of $[n].$  If one part of a split has only a single element, we call that split \textcolor{def}{\emph{trivial.}} A \textcolor{def}{\emph{split system}}  is a set $s$ of splits of $[n]$ which contains all the trivial splits.  There are $\displaystyle{2^{(2^{n-1}-n-1)}}$ such systems, for $n\ge 3$. This sequence begins 1, 8, 1024,..., and is found as A076688 in \cite{oeis}. We say a split system $s$  \textcolor{def}{\emph{refines}} another split system $s'$ when $s \supset s'$.

In this paper all graphs are simple (no multi-edges) and connected.
\begin{definition}
An  \textcolor{def}{\emph{(unrooted) phylogenetic network}} on $[n]$ is a simple connected graph with $n$ degree-1 nodes labeled bijectively by the elements of $[n]$, and all other nodes unlabeled and with degree larger than 2.
\end{definition} 

A split $A|B$ is \textcolor{def}{\emph{displayed}} by such a phylogenetic network $N$ if there is at least one subset of edges of $N$ whose deletion (keeping all nodes) results in two connected components  with respective labeled nodes the two parts of that split. We call that collection of edges a \textcolor{def}{\emph{minimal cut}} for the split when it contains no proper subset producing the same split.  A \textcolor{def}{\emph{bridge}} is a single edge which displays a split. A \textcolor{def}{\emph{trivial bridge}} displays a trivial split. A \textcolor{def}{\emph {phylogenetic tree}} is a cycle-free network, so every edge is a bridge. 
\begin{figure}
    \centering
   \includegraphics[width=\textwidth]{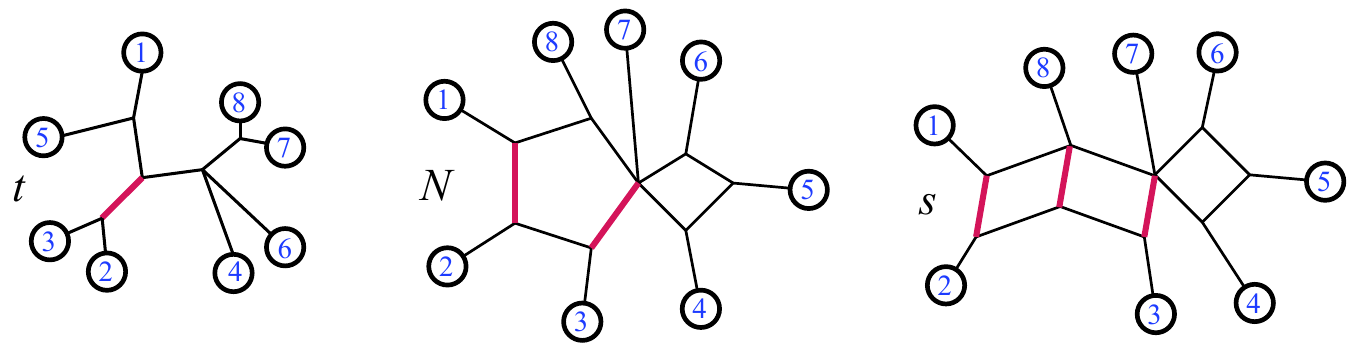}
    \caption{In a phylogenetic tree $t$, on the left, splits are always single edges. The highlighted edge is the split $\{2,3\}|\{1,4,5,6,7,8\}.$ That same split is a pair of edges making a minimal cut in the 1-nested phylogenetic network $N$, center. Finally on the right, that same split is a set of parallel edges in a circular split network $s$.}
    \label{splits}
\end{figure}

The unrooted phylogenetic networks are often classified by complexity. This classification is usually based on the \textcolor{def}{\emph{bridge-free}} components, that is, the subgraphs (with more than one node) which remain after removing all the bridges. The following
 is defined in \cite{gambette-huber}: 
 \begin{definition}
 An unrooted phylogenetic network is called \textcolor{def}{\emph{1-nested}}  when each edge is contained in at most one cycle (recall that cycles do not revisit nodes), and all cycles are of length greater than 3 edges. This allows for multiple cycles in a bridge-free component to share a node, but not an edge.
 \end{definition}

\subsection{Ordering and Counting: 1-nested Phylogenetic Networks}\label{order}

A split system  is determined by a 1-nested phylogenetic network, but not uniquely. 
We begin by considering two such networks as equivalent if they display the same set of splits. Otherwise they are ordered by inclusion.
\begin{definition}
The poset of 1-nested phylogenetic networks is defined as follows: $N\le N'$ precisely when all the splits displayed  by $N$ are also displayed by $N'.$ Also $N \cong N'$ if $N \le N'$ and $N'\le N.$   
\end{definition}

  Clearly this is a partial order by construction, since the the networks are equivalent precisely when their sets of splits are equal. An example of the inequality is in Figure~\ref{fig:phylonet_ineq}, and an equivalence is shown in Figure~\ref{fig:mo_examplo}. 
\begin{figure}
    \centering
    \includegraphics[width=5in]{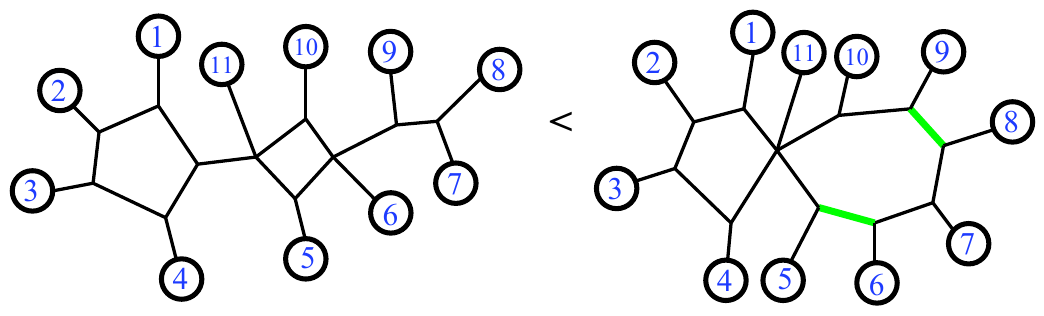}
    \caption{The phylogenetic network on the left displays a proper subset of the splits displayed by the network on the right. For instance the highlighted split $\{1,2,3,4,5,9,10,11\}|\{6,7,8\}$ is only displayed on the right.}
    \label{fig:phylonet_ineq}
\end{figure}  

Two equivalent 1-nested networks are related by the collapse or growth of specific edges: in fact any non-trivial bridge directly attached to a cycle can be collapsed without changing the displayed splits. See Figure~\ref{fig:mo_examplo} for examples. This allows a convenient choice of representative: we can put in place as many nontrivial bridges as possible without adding any splits. Doing so illuminates a bijective correspondence between the equivalence classes of 1-nested phylogenetic trees and $PC$-trees, as indicated in \cite{gambette-huber}. The latter are  leaf-labeled trees with two types of internal nodes: the $P$-nodes are tree-like in that they allow all cyclic permutations of their (3 or more) attached edges, but the $C$-nodes each have a given cyclic order of their (4 or more) attached edges. 
\begin{thm}\label{pctree}
 Split systems on $[n]$ displayed by 1-nested phylogenetic networks are in bijection with $PC$-trees with leaves $[n].$  
\end{thm}
\begin{proof}
 The correspondence is by taking a representative $N$ of the equivalence class of 1-nested phylogenetic networks that has as many non-trivial bridges as possible. Then each cycle of $N$ is replaced by a $C$-node with the corresponding cyclic order. The remaining internal nodes are the $P$-nodes. 
\end{proof}
An example of the correspondence is in Figure~\ref{pccorr}. $C$-nodes are drawn as small squares, and $P$-nodes as circles. The network pictures are convenient for visualizing the split systems, but the $PC$-trees are easier to count. (It is the same problem as enumerating the $PQ$-trees, which are the rooted version in bijection with the $PC$-trees as shown in \cite{kleinman}.) This solution is partially found in \cite{bergeron_book}, and the  
exponential generating function $f(x)$ is given in \cite{oeis}:  $$f(x) = g^{-1}(x), \text{ for }
g(x)= \frac{x^3+4x^2-2x-2}{2(x-1)} - e^x.$$

Thus the numbers of 1-nested phylogenetic networks on $[n]$, up to equivalent sets of displayed splits, are 1, 7, 68, 941, 16657, 360151, ... as listed in OEIS entry A136629.

\begin{figure}
    \centering
    \includegraphics[width=\textwidth]{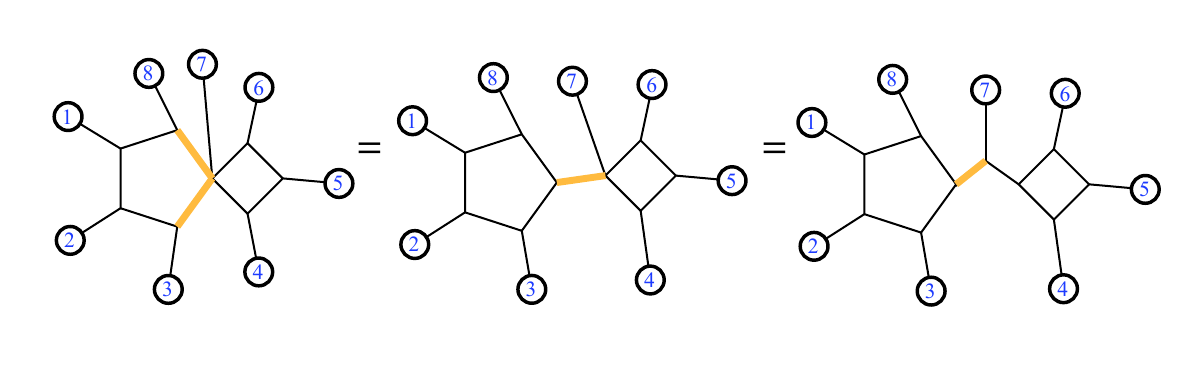}
    \caption{A trio of equivalent 1-nested phylogenetic networks, all representing the same set of splits. The highlighted edges display the same split in each network.}
    \label{fig:mo_examplo}
\end{figure}

\begin{figure}
    \centering
    \includegraphics[width=3.75in]{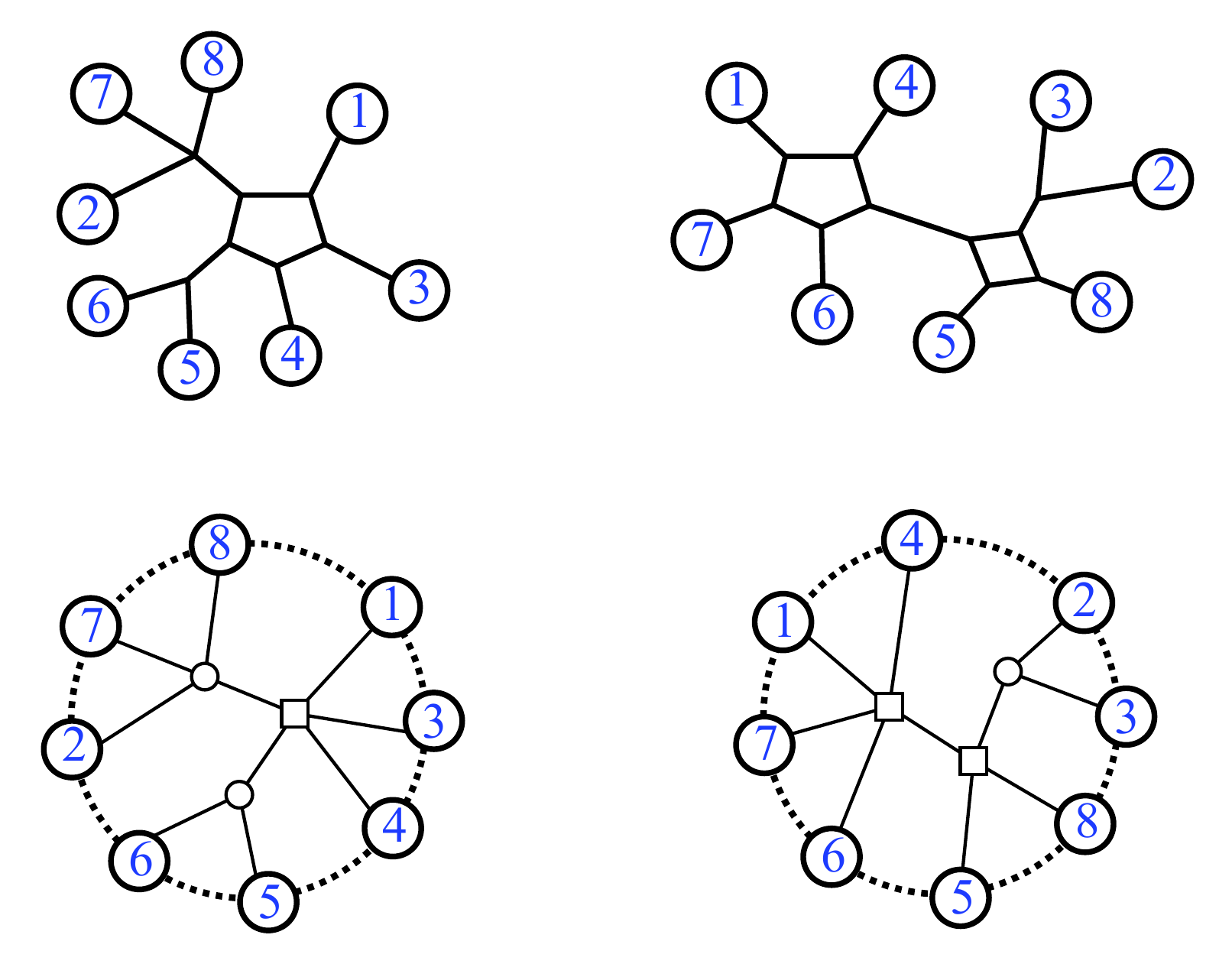}
    \caption{Two 1-nested phylogenetic networks with their corresponding PC-trees beneath.}
    \label{pccorr}
\end{figure}

 A \textcolor{def}{\emph{binary}} phylogenetic network is one in which the unlabeled nodes each have degree 3.
In \cite{newgamb} the authors define \textcolor{def}{\emph{binary level-k}} networks to be those which require no more than $k$ edges to be removed from each biconnected component  before the result is a phylogenetic tree. Here we restrict our attention at first to level-0 and level-1. 
Note that level-0, or 0-nested, networks are a special sort of 1-nested networks:  phylogenetic trees.  Binary level-1 networks are the same as binary 1-nested networks. These are counted in \cite{durell-forcey}, where a formula for the number of binary 1-nested networks on $[n]$ with $k$ nontrivial bridges is found: $${n-3 \choose k}\frac{(n+k-1)!}{(2k+2)!!}$$
Especially note the cases of $k=0,$ where the number becomes $((n-1)!)/2,$ counting the number of cyclic orders; and $k=n-3,$ where the number is $(2n-5)!!$ counting the number of phylogenetic trees. The following is immediate upon inspection of the $PC$-trees that correspond to the binary 1-nested networks
\begin{corollary}
Binary 1-nested  networks are in bijection with $PC$-trees for which the all nodes are cyclic (class $C$) except for the nodes of degree 3 (which are permutable, class $P$.)
\end{corollary}

\subsection{Split networks}
We also will consider another generalization of a phylogenetic tree, in which each split corresponds to a unique set of edges. 

\begin{definition}
A  \textcolor{def}{\emph{split network}} displaying a split system $s$ on $[n]$ is an embedding in Euclidean space of a simple connected graph, also called $s$, with the following:
\begin{enumerate}\item[i.] exactly $n$ degree 1 nodes called leaves labeled by $[n]$, and the other nodes
unlabeled; 
\item[ii.] the set of edges partitioned into classes, one class for each split $A|B$ in the system. It is required that for any two nodes:  the set of edges on a shortest path (of fewest edges) between them intersects each split-class in at most one edge, and that the set of splits thus traversed is the same for any shortest path between those two nodes.
\item[iii.] The class of edges corresponding to a split comprises a minimal cut displaying that split:  deletion of those edges (keeping all nodes) results in two connected components with respective labeled nodes the two parts of that split. 
\end{enumerate}
\end{definition}
Typically each class of edges is embedded as a set of equal length parallel line segments. Alternate definitions use colors;  the edges in a split-class are colored alike, as in, \cite{basic}, \cite{steelphyl}. The resulting graph will be bipartite. 

\begin{figure}
    \centering
    \includegraphics[width=5in]{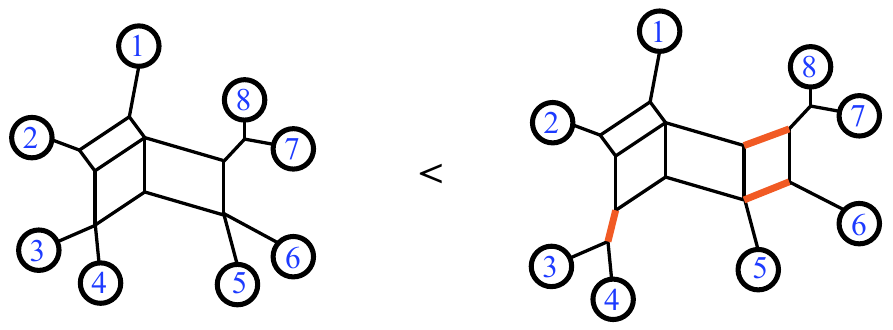}
    \caption{The split network on the left displays a proper subset of the splits displayed by the network on the right. Specifically the two highlighted splits are only displayed on the right.}
    \label{fig:splitnet_ineq}
\end{figure}
Several different split networks may often  be drawn for the same split system, but we consider them equivalent as long as they represent the same splits. Note that in contrast to unrooted phylogenetic networks, the only sort of minimal cut of a split network that is said to display a split is one of the classes of parallel edges. Other cuts are ignored. The fact that for every split system it is possible to construct a split network is due to Buneman. In fact, the most common choice of a representative split network for a given split system is called the \textcolor{def}{\emph{Buneman graph}}, achieved with Buneman's algorithm \cite{basic}. It is a median graph, and is thus a retract of a hypercube 1-skeleton. However in this paper we will only need to consider certain split networks that can be drawn on the plane.

\begin{definition}
A \textcolor{def}{\emph{circular split system}} is a split system which allows the embedding of representative split networks in the plane, with the labeled nodes all on the exterior, and thus arranged in a circular order.
\end{definition}
   \textcolor{def}{\emph{Twisting}} the diagram around a bridge (reflecting one side through the line of the bridge), or around a cut-point node, does not change the list of splits.  Any cyclic order of the leaves allowing an embedding of a split network in the plane is said to be \textcolor{def}{\emph{consistent}} with that system. 

  \begin{figure}
     \centering
     \includegraphics[width=5.25 in]{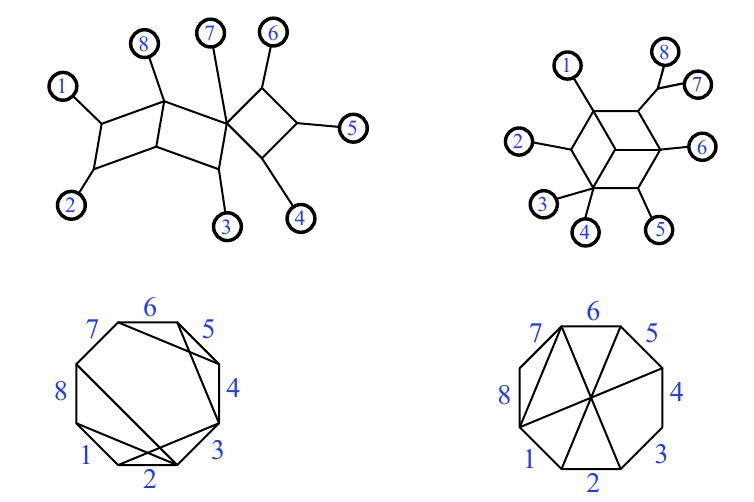}
     \caption{Circular split systems with corresponding polygonal diagrams beneath each.}
     \label{penta}
 \end{figure}

 Alternatively, we can define a circular split system on $[n]$ as follows: we can see the splits by 1) labelling the sides of an $n$-gon with $[n]$ and 2) drawing diagonals for each split. That this is an equivalent description to the above is well known, shown for instance in \cite{dev-petti}. Examples are shown in Figure~\ref{penta}.  

We introduce here a subclass of circular split networks which play an important role as the range of the neighbor-net algorithm when restricted to metrics arising from 1-nested phylogenetic networks. 
\begin{definition}
An \textcolor{def}{\emph{outer-path circular split system}} is a split system whose representative circular split networks have shortest paths between pairs of leaves which can all be chosen to lie on the exterior of the diagram, that is, using only edges adjacent to the exterior. 
\end{definition}
Since the shortest paths are all the same length, this implies that outer-path split networks have no \emph{shortcut}, that is, their is no path between leaves through the  interior of the diagram that is strictly shorter than any path on the exterior. For examples, see Figure~\ref{shorty}.
\begin{figure}
    \centering
    \includegraphics[width=\textwidth]{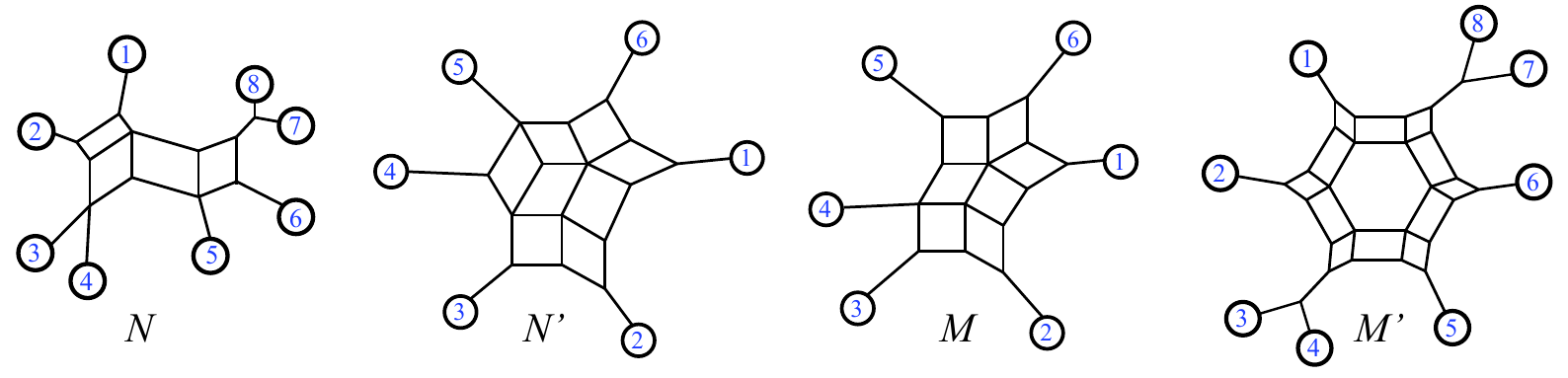}
    \caption{Two outer-path circular split networks on the left, $N$ and $N'$, have all shortest paths represented by exterior paths. Two non-outer-path circular split networks on the right, $M$ and $M'$ have representations which include shortcuts: for instance the path from leaf 1 to 4 in $M$ and the path from 2 to 6 in $M'$. }
    \label{shorty}
\end{figure}

\subsection{Ordering and Counting: Circular Split Networks} In \cite{terhorst} the circular split networks are studied as a poset, the Kalmanson complex. In that source the numbers of circular split networks are seen for $n=3,4,5,$ and $6;$  giving respectively 1, 7, 218, 20816. Some of these totals are enumerated by underlying type of network in \cite{dev-petti}, and by dimension in the simplex, in \cite{terhorst}. This sequence is bounded below for each $n$ by the numbers of 1-nested networks. Circular split networks represent at least as many split systems as do 1-nested networks, for each $[n],$ since the former have a stricter definition of displayed splits. 

Circular split networks are ordered by inclusion of their sets of splits. Also two circular split networks are equivalent if they display the same set of splits. 
\begin{definition}
The poset of circular split networks is defined as follows: $s\le s'$ precisely when all the splits displayed  by $s$ are also displayed by $s'.$ Also $s \cong s'$ if $s \le s'$ and $s'\le s.$   

\end{definition}
\begin{definition}
The poset of outer-path circular split networks is defined as the full sub-poset of circular split networks; the outer-path circular split networks and all the relations that exist between them. 
\end{definition}
An example of the inequality is in Figure~\ref{fig:splitnet_ineq}, and an equivalence is shown in Figure~\ref{fig:mo_examplo_circ}.

\begin{figure}
    \centering
    \includegraphics[width=6in]{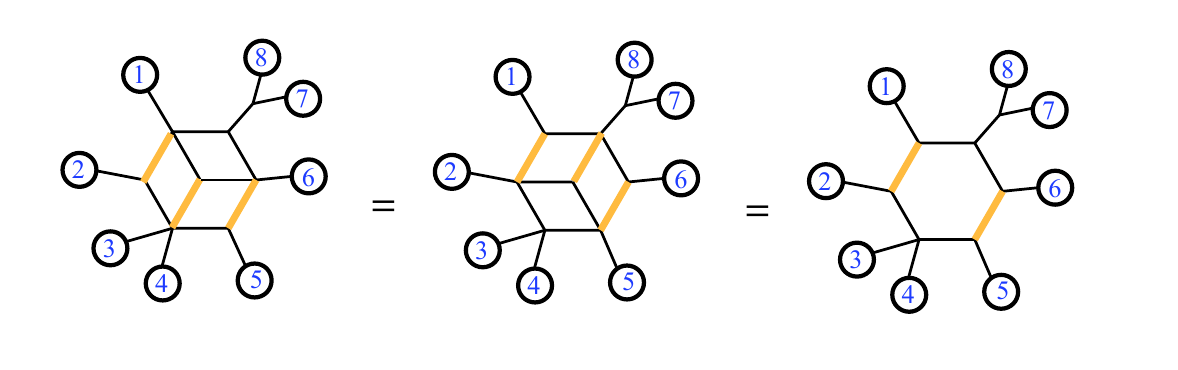}
    \caption{A trio of equivalent split networks, all three representing the same set of splits. The highlighted edges display the same split in each network.}
    \label{fig:mo_examplo_circ}
\end{figure}
Now we list a few lemmas that will be useful in the next section.
\begin{lemma}
Given a circular split network $s$,
the nodes and edges adjacent to the exterior of the graph are a subgraph which is invariant: that is, this \emph{exterior subgraph} will be identical to the exterior subgraph of any circular split network representing the same set of splits as $s$. 
\end{lemma}
\begin{proof} 
Bridges will clearly be part of the exterior subgraph, and present in any representing network of $s$. For non-bridge splits, consider the polygonal representation of $s$.
Non-trivial bridges of $s$ correspond to non-crossed diagonals of the polygonal representation, and cut-point nodes correspond to regions that are diagonal free.    
Each bridge-free (and cut-point-node-free) component of $s$ corresponds to a set of mutually crossing diagonals, and this collection of splits is the same in any equivalent polygonal diagram. Thus the same collections of non-bridge splits will be contained each in a single (bridge and cut-point-node free) component of any network equivalent to $s.$  Each of these non-bridge splits is displayed in $s$ with a parallel class of edges, two of which are adjacent to the exterior. The only variation of the order of leaves on the exterior will be due to twisting around bridges or cut-point nodes. Thus the leaves and exterior edges will always form the same subgraph, regardless of representing network.  
\end{proof}
 For examples see Figure~\ref{fig:mo_examplo_circ}. The exterior subgraph will be a series of cycles of even length, connected  by  cut-point nodes, nontrivial bridges, and trivial bridges to the leaves. In fact the exterior subgraph of $s$ is a circular split network itself, displaying the same system as $s$. (Typically, however, more interior edges are shown since parallelograms can help make the splits visually identifiable.) 

The following is immediate, since adding splits to a network only subtracts from the set of circular orders consistent with that network.
\begin{lemma}\label{refine}
 For split networks $s \subseteq s'$, if $c$ is a circular order consistent with $s'$ then $c$ is consistent with $s.$
\end{lemma}

The 1-nested phylogenetic networks can also be drawn on the plane, with their leaves on the exterior. Just as for circular split networks, twisting the diagram around a bridge (reflecting one side through the line of the bridge), or around a cut-point node, does not change the list of splits.  Again any circular order of the leaves allowing the representation is called consistent with that network. 

\begin{lemma}\label{contig}
 A split $A|B$ is displayed by a 1-nested phylogenetic network $N$ (or a circular split network $s$) if and only if  every  circular order consistent with $N$ (consistent with $s$) has both parts $A$ and $B$ contiguous. 
\end{lemma}
\begin{proof}
If a split is displayed, then the graph immediately exhibits a circular order with both parts contiguous. Furthermore, no twisting around a bridge or cut-point node can then separate the elements of $A$ in the resulting circular order. The converse follows since if $A|B$ not displayed, then elements of $A$ must either be found: (1) on both sides of a bridge, with elements of $B$ on both sides as well, or (2) on both sides of a cut-point node, with elements of $B$ on both sides as well, or (3) neither (1) nor (2), but as elements all attached to the same cycle in the graph---either directly via trivial bridges or as larger subgraphs entirely labeled by elements of $A.$ In the cases (1) or (2), a twist of the bridge or cut-point node results in  a consistent cyclic order with $A$ not contiguous. In case (3) the elements of $A$ cannot be contiguous in any circular order (else there would be a pair of edges of that cycle displaying the split $A|B$). 
\end{proof}

\section{Functions}\label{functions}

Although circular split networks and 1-nested phylogenetic networks are both planar with leaves on the exterior, they display splits in distinctly different ways.
As shown in \cite{gambette-huber}, any split system displayed by a 1-nested phylogenetic network can be displayed by a circular split network, but not the other way around. Instead, we consider maps between the two posets. First there is a map from phylogenetic networks to split networks:
\begin{definition}
For a 1-nested phylogenetic network $N$ we define $\Sigma(N)$ to be the circular split system made up of the splits displayed by $N.$\end{definition} 

In \cite{gambette-huber} it is shown that $\Sigma(N)$ can be displayed by a circular network, also referred to as $\Sigma(N)$.
Since $N \cong N'$ precisely when they display the same set of splits, $\Sigma$ is well defined. (Indeed $\Sigma$ is the same function as $\beta,$ defined in \cite{kleinman} as giving the set of splits displayed by a $PC$-tree.)  An algorithm for drawing a representing network of $\Sigma(N)$ is also presented in  \cite{gambette-huber}. First, cycles of length 4 are each replaced by a parallelogram. For $m \ge 5$, each $m$-cycle is replaced by an $m$-\emph{marguerite}: a collection of exactly $m^2-4m$ parallelograms arranged in a circle, each sharing sides with two neighbors, specifically organized as follows: each node of the original $m$-cycle is replaced by a rhombus, and then each edge of the cycle is replaced by $m-5$ parallelograms in a row. The  rows are attached to the  rhombi along adjacent edges of each rhombus, so that the whole arrangement has $m(m-5)$ sides on the interior of the original $m$-cycle, and $m(m-3)$ sides on the exterior. Bridges are attached to the $m$ remaining degree-2 vertices, one at each of the rhombi that replaced the original $m$ nodes of the cycle. Examples of representations of $\Sigma(N)$ are seen in Figure~\ref{fig:sigma_map} and later in Figure~\ref{fig:bigo_k1}.

The properties of $\Sigma$ are extensively discussed in \cite{gambette-huber}, including the fact that it takes 1-nested phylogenetic networks (displaying at least all the trivial splits) to circular split networks (also with all trivial splits included.) A bridge in $N$ is still present as a bridge in $\Sigma(N).$ Notice that while $\Sigma$ preserves bridges, it may also introduce new bridges: a cut-point node in $N$ can become a set of bridges in the image.   By its definition the function $\Sigma$ respects refinement; it is a monotone poset function: if $N \le N'$ then $\Sigma(N)\le \Sigma(N').$  

\begin{figure}
    \centering
    \includegraphics[width=5in]{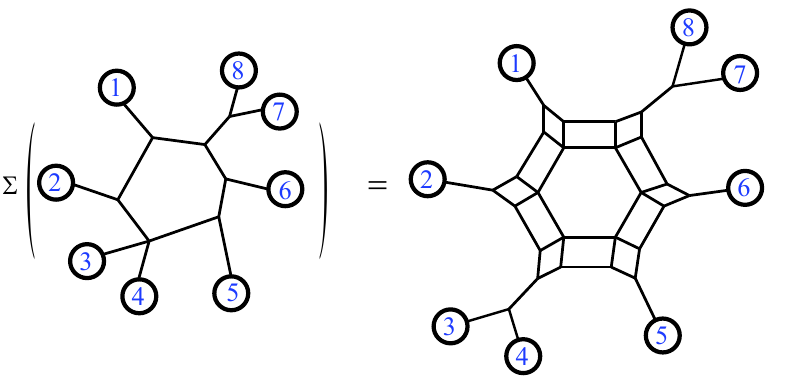}
    \caption{Example of the function $\Sigma.$}
    \label{fig:sigma_map}
\end{figure}

Next we define a function that takes a circular split network to a 1-nested phylogenetic network.  The function is shown to exist in \cite{gambette-huber}, and described on the split networks which are images of the function $\Sigma.$ In \cite{durell-forcey} we define the general function $L$ as follows: 
\begin{definition}
  Recall that the nodes and edges adjacent to the exterior of a circular split network are an invariant subgraph for the split system. Define $L(s)$ to be the \emph{smoothed exterior subgraph} of $s.$\end{definition}
  
  In other words, we construct the network $L(s)$ from a split system $s$ by beginning with a split network diagram  of $s$ and  considering the diagram as a planar drawing of its underlying planar graph, with  leaves on the exterior.  Then 1) delete all the edges that are not adjacent to the exterior of that graph, and 2) smooth away any resulting degree-2 nodes---delete the node but join the two adjacent edges to make one edge. 

\begin{figure}
    \centering
   \includegraphics[width=5in]{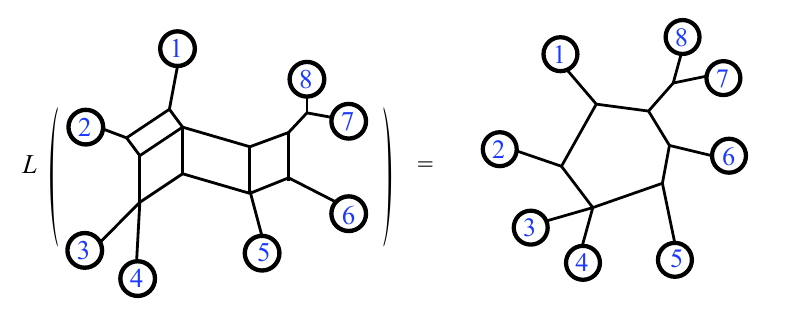}
    \caption{Example of the function $L.$}
    \label{fig:Lmap}
\end{figure}

Note that by its construction, $L$ preserves bridges and cut-point nodes. We also have the monotone property:
\begin{lemma}
For two circular split networks $s\le s'$, we have $\L(s) \le L(s').$
\end{lemma}
\begin{proof}
By construction, $L$ preserves the circular order and bridges of the split network diagram. If $s\le s'$ then there are some splits of $s'$, each drawn as a set of parallel edges, which can be shrunk to length 0 to see a representative split network of $s.$ These collapsed edges include some adjacent to the exterior of the diagram, and thus $L(s)$ will display a subset of the splits displayed by $L(s').$
\end{proof}

\subsection{Galois connections}

When restricted to phylogenetic trees, the functions $L$ and $\Sigma$ are both the identity. In general however, we have the following:

\begin{figure}
    \centering
    \includegraphics[width=\textwidth]{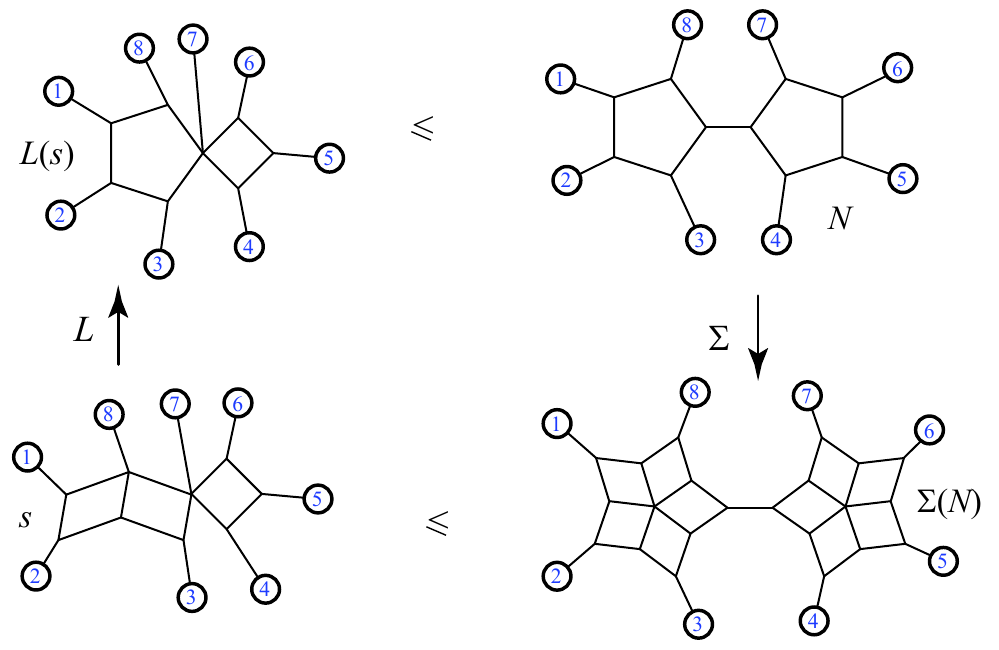}
    \caption{Illustration of the Galois connection between unweighted phylogenetic networks and split networks.}
    \label{fig:galois_d}
\end{figure}

\begin{thm}\label{unwgalois}
For any circular split system $s$ and 1-nested network $N,$ 

$$L(s) \le  N \text{ if and only if } s \le \Sigma(N) $$

That is, $L$ and $\Sigma$ form a Galois connection in which $L$ is the lower and $\Sigma$ the upper adjoint.  
\end{thm}
\begin{proof}
For an arbitrary $s$, choose any $N$ such that $L(s) \le N$.  Consider a split $A|B \in s.$ We see that $A|B$ is displayed by $L(s)$ by construction: since $A|B$ is displayed by parallel edges of $s$ which include one or two edges on the exterior, which will be incorporated into one or two edges of $L(s)$, in turn comprising a minimal cut also displaying $A|B.$   The given inequality implies by definition that every split of $L(s)$ is displayed by $N.$ Thus $ A|B \in \Sigma(N),$ and we have $s \subseteq \Sigma(N).$


For the converse: for an arbitrary $N$, choose any $s$ such that $s \subseteq \Sigma(N).$  We show the contrapositive of the desired conclusion: we claim that if $A|B$ is not in the set of splits displayed by $N$ then $A|B$ is not displayed by $L(s)$ . Via Lemma~\ref{contig}, if the split $A|B$ is not displayed by  $N$ then there is  a circular order $c$ consistent with $N$ with  $A$ not contiguous in $c$. Then that same circular order is consistent with $\Sigma(N)$ since $\Sigma$ preserves any consistent $c.$ Then $c$ is consistent with $s$ by Lemma~\ref{refine}.  However then $c$ is also consistent with $L(s)$ and recall that $A$ is not contiguous in $c$. Then, the same lemma implies that $L(s)$ does not display $A|B$. Therefore $L(s) \le N.$ 



\end{proof}
Notice that the composition $\Sigma \circ L$  is increasing. 
When the upper adjoint is injective, we call the Galois connection a reflection, and the lower adjoint is implied to be surjective.  

\begin{thm}\label{reflection}
 The Galois connection via $L$ and $\Sigma$ is a reflection (but not a poset isomorphism).
\end{thm}
\begin{proof}
The upper adjoint $\Sigma$ is injective simply because its output is defined to be the set of splits displayed by the input $N$, and two 1-nested networks are equivalent precisely when they display the same set of splits. Next we show that $L$ is not injective. This is clear when we look at two split networks which both have a bridge-free portion with 5 or more bridges attached (leading to leaves or other bridge-free portions.) Allow those leaves and other bridge-free portions  of the two respective split networks to be identical. Then we can have different sets of splits displayed by the portion we are focused on, and yet both are mapped to the same 1-nested network by $L$.   
\end{proof}
As corollaries from Galois theory \cite{primer} we see that $L$ is surjective, (but $\Sigma$ is not surjective) and that $L\circ\Sigma$ is the identity map. Figure~\ref{fig:galois_c} exhibits examples of these facts.
\begin{figure}
    \centering
    \includegraphics[width=\textwidth]{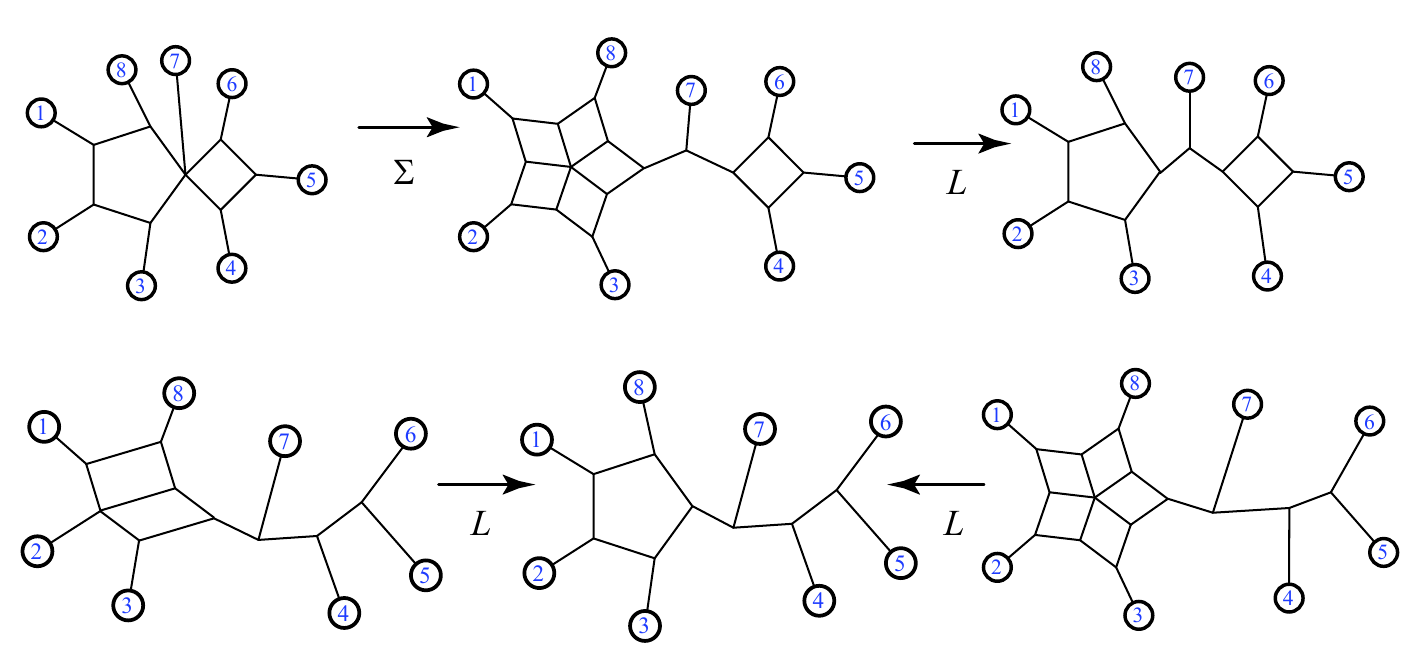}
    \caption{Above, demonstration of the fact that $L\circ\Sigma$ is the identity map; compare to Figure~\ref{fig:mo_examplo}. Below, example of the fact that $L$ is not an injection.}
    \label{fig:galois_c}
\end{figure}

\section{Weighting}\label{sec_weight}
Weighted phylogenetic trees have non-negative real number assigned to their branches, often representing the genetic distance between the two nodes. A weight of 0 can mean the edge is collapsed, and the resulting space of trees, called BHV${}_n$, is studied in \cite{bhv}. Now we may generalize weighted trees with weighted networks  in two distinct ways: by assigning non-negative real numbers to splits or to edges.  
\begin{definition}
A \textcolor{def}{\emph{weighted phylogenetic network}} $N$ has non-negative real numbers assigned to its edges, described  by a weight function $w_N.$ 
\end{definition}

\begin{definition}
A \textcolor{def}{\emph{weighted split network}} $s$ has non-negative weights assigned to each split, by a weight function $w_s$. Equivalently, every edge in a parallel class of $s$ has the same weight.
\end{definition}

\begin{definition}
For a weighted phylogenetic network $N$,  or a weighted split network $s$, we denote by $\overline{N}$, respectively $\overline{s}$, the unweighted networks found by forgetting the weights. 
\end{definition}
 A weight of zero often means that we can consider that edge (or split) as collapsed; in the topological picture  this results in two or more networks being identified. For split networks, the resulting quotient space CSN${}_n$ of all weighted circular split networks is studied in \cite{dev-petti}.   In \cite{kleinman} the authors consider weighted PC-trees, but there the weights are assigned to the splits---as opposed to the edges as for 1-nested phylogenetic networks here.
 
The contrast in weighting definitions---weighted edges in a phylogenetic network vs weighted splits in a split network---is explained from the perspective of phylogenomics.  The weight of an edge corresponds to the hypothetical difference in the DNA sampled at the beginning and end of a period of time. The weight of a split however, corresponds to a difference that is common to all pairs of taxa found on either side. In the application to a real set of taxa, the data collected is distilled into a pairwise difference function.  

A pairwise distance function assigns a non-negative real number to each pair of values from $[n]$. We call the lexicographically listed outputs for distinct pairs a \textcolor{def}{\emph{distance vector}}  $\mathbf{d}$, with entries denoted $d_{ij} = \mathbf{d}(i, j)= \mathbf{d}(j, i)$ for each pair of taxa $i \neq j \in [n]$ (also known as a dissimilarity matrix, or discrete metric when obeying the metric axioms.)

There are two special kinds of distance vector we consider. When the distance vector is \textcolor{def}{\emph{additive}} it means that for all $i,j,k,l\in [n],$ $\mathbf{d}$ obeys the  \textcolor{def}{\emph{four-point}} condition: 

$$d_{ij} + d_{kl} \le \max\{d_{ik}+d_{jl}, d_{il}+d_{jk}\}.$$

When the distance vector is \textcolor{def}{\emph{Kalmanson}}, or \textcolor{def}{\emph{circular decomposable}} it means there exists a cyclic order of $[n]$ such that for any subsequence $(i,j,k,l)$ of that order, $\mathbf{d}$ obeys this condition: $$\max\{d_{ij}+d_{kl}, d_{jk}+d_{il}\} \le d_{ik}+d_{jl}.$$

 \begin{definition}
Given  a weighted split system $s$ on $[n]$ we can derive a metric $\mathbf{d}_s$ on $[n],$ 
$$\mathbf{d}_s(i,j) = \sum_{i\in A, j\in B} w_s(A|B)$$ 
where the sum is over all splits of $s$ with $i$ in one part and $j$ in the other. The metric is often referred to as the distance vector $\mathbf{d}_s.$
 \end{definition}
 
It is well known that additive metrics are represented uniquely by weighted phylogenetic trees. That is, $\mathbf{d}$ is additive if and only if $\mathbf{d}$ = $\mathbf{d}_s$ for $s$ a unique weighted phylogenetic tree.
Furthermore, it is well known that Kalmanson metrics are represented uniquely by weighted circular split networks. Specifically, from \cite{steelphyl} we have the following: \begin{lemma}\label{kal}
A distance vector $\mathbf{d}$ is Kalmanson with respect to a circular order $c$ if and only if $\mathbf{d}$ = $\mathbf{d}_s$ for $s$ a unique weighted circular split system $s$, (not necessarily containing all trivial splits) with each split $A|B$ of $s$ having both parts contiguous in that circular order $c$.
\end{lemma}

\begin{definition}
We also define a distance vector $\mathbf{d}_N$ for a weighted 1-nested phylogenetic network $N,$ where 
$$\mathbf{d}_N(i,j) = \min_p\{\sum_{e \in p} w_N(e)~|~p  \text{ is a path connecting } i,j\}$$ where the minimum is over paths $p$ from leaf $i$ to leaf $j,$ and each sum is over edges in one of those paths.
\end{definition}

\subsection{Ordering} For weighted networks of either variety, we restrict the partial ordering so that only networks with identical distance vectors are possibly comparable. 
\begin{definition}
For $N$ and $N'$  two 1-nested weighted phylogenetic networks we say $N \le N'$ when $d_N = d_{N'}$  and the splits displayed by $N$ are a subset of those displayed by $N'.$  
\end{definition} 
Note  that by definition we have the following:
\begin{lemma}\label{overl}
If for two weighted phylogenetic networks, we have $N\le N'$ then $\overline{N} \le \overline{N'},$ for the unweighted versions.
\end{lemma}
For weighted circular split networks, the analogous restriction of the poset makes it trivial. 
\begin{definition}\label{defw}
For $s$ and $s'$  two weighted circular split networks we say $s \le s'$ when $d_s = d_{s'}$  and the splits displayed by $s$ are a subset of those displayed by $s'.$ However, every relation in this case is an equality, since the Kalmanson metrics are uniquely displayed 
\end{definition}

\subsection{Functions} Now we define functions between the weighted split networks and the weighted phylogenetic networks. As previously explained in \cite{durell-forcey}, we begin by extending the function $L$ to a weighted version $L_w.$

\begin{definition}
For a weighted circular split network $s$ we define $L_w(s)$ to be the 1-nested phylogenetic network $L(\overline{s})$ (the smoothed exterior subgraph of the unweighted version of $s$), with weighted edges. The weight of an edge in the image is found by summing the  weights of splits which contribute to that edge. Let $p_s(e)$ be the set of splits $A|B$ of $s$, such that $A|B$ is represented by edges in $s$ one of which is used to form the edge $e$ in $L(s)$. If $w_s$ is the weight function on $s$ then the weight function on $L_w(s)$ is:

$$w_{L_w(s)}(e) = \sum_{A|B \in p_s(e)} w_s(A|B).$$  
\end{definition}
By this definition we have the following: \begin{lemma}\label{comm}
$\overline{L_w(s)} = L(\overline{s}).$
\end{lemma} 
For examples see Figure~\ref{swlwuno}, as well as Figure~\ref{fig:galois_w}.
\begin{figure}
    \centering
    \includegraphics[width=\textwidth]{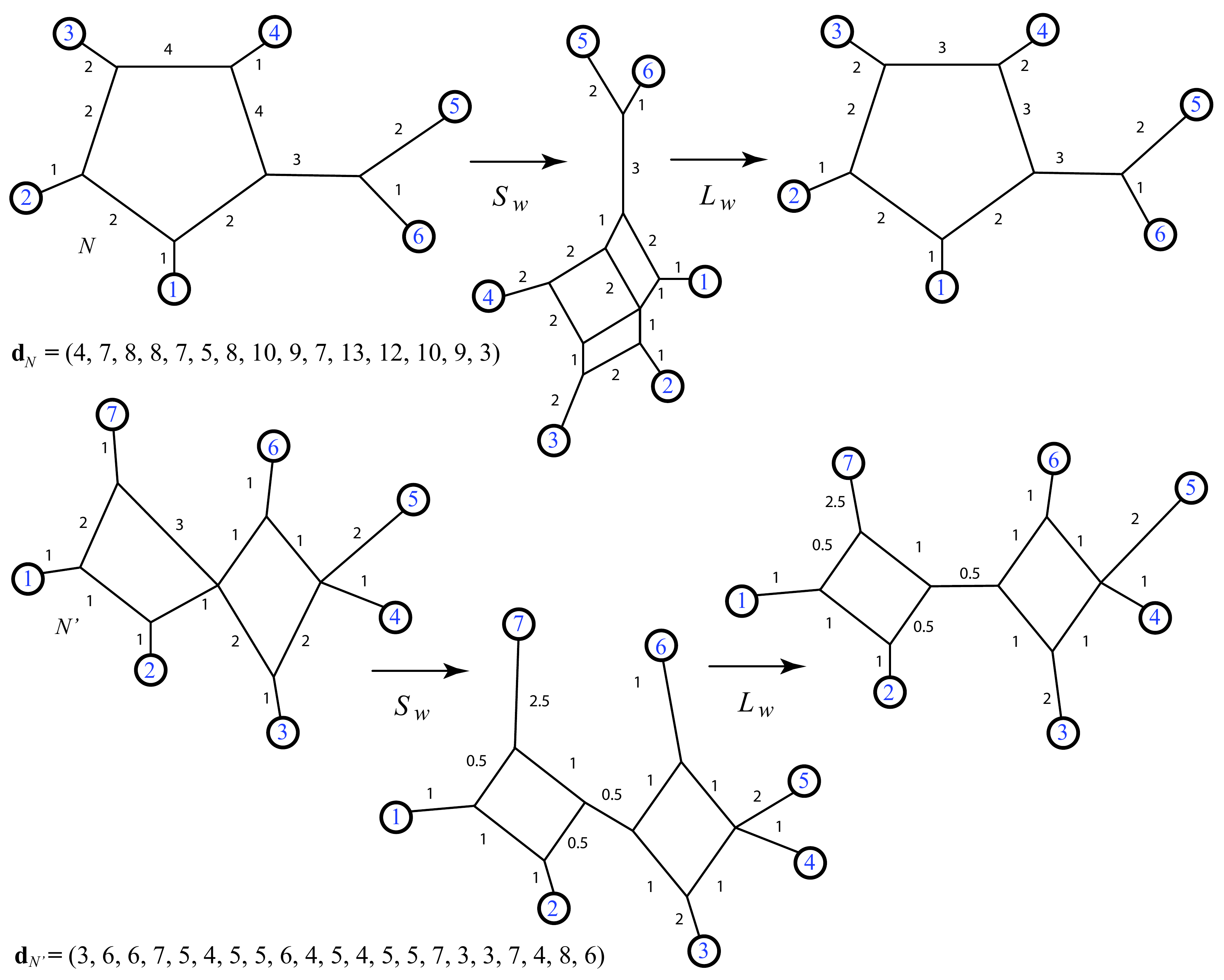}
    \caption{Examples of the action of both $L_w$ and $S_w$.}
    \label{swlwuno}
\end{figure}

Taking a weighted 1-nested phylogenetic network to a circular split network is also described in \cite{durell-forcey} Here we extend the definition to all weighted planar phylogenetic networks.
\begin{definition}
Given a weighted unrooted phylogenetic network $N$ that can be drawn on the plane with leaves on the exterior, we define $S_w(N)$ to be the circular split network $\mathcal{N}(\mathbf{d}_N).$ Here $\mathcal{N}$ is the neighbor-net algorithm defined by \cite{Bryant2007}.  
\end{definition}

For examples see Figure~\ref{swlwuno}, as well as Figure~\ref{fig:galois_w}.

We could also define $S_w(N)$ to be the unique weighted circular split network with the same distance vector as $N$. The following lemma is generalized slightly from \cite{durell-forcey} to cover all planar networks.

\begin{lemma}
Given a weighted planar phylogenetic network $N,$ there is a unique circular weighted split
system $s = S_w(N)$ which has the same associated distance vector as $N$. That is, $\mathbf{d}_N$ = $\mathbf{d}_s$.
\end{lemma}

\begin{proof}
 First we show that $\mathbf{d}_N$ obeys the Kalmanson condition: there exists a
circular ordering of $[n]$ such that for all $1 \le i < j < k < l \le n$ in that ordering,

$$\max\{\mathbf{d}_N(i,j)+\mathbf{d}_N(k,l),\mathbf{d}_N(j,k)+\mathbf{d}_N(i,l)\} \le \mathbf{d}_N(i,k)+\mathbf{d}_N(j,l).$$

The circular ordering that meets our specifications is just any choice of one of the circular
orderings consistent with $N$. Our network $N$ is planar, so the edges are drawn with no crossings.
The two paths involved on the right hand side of the condition intersect each other. Then since
the leaves are on the exterior, the four paths involved on the left hand side of the condition are
each bounded above in length by a path made by following first one intersecting path and then
the other, (switching at the crossroads, after their shared portion.) Two paths in a sum on the left
hand side of the condition can at most use exactly all of both the intersecting paths, so that the
inequality is guaranteed.
It is well known that for any Kalmanson metric $\mathbf{d}_N$ there exists a unique weighted split system $s$
whose weighting gives that metric: $\mathbf{d}_N$ = $\mathbf{d}_s$. To actually calculate this split system, the algorithm
neighbor-net can be used; since it is guaranteed to return the unique answer for any Kalmanson
metric \cite{steelphyl}. \\\\
\end{proof}

In order to see that $S_w$ has the correct range, we must check that the map $S_w$ takes 1-nested phylogenetic networks with labeled leaves to circular split networks that contain all the trivial splits. In fact, we show the following:
\begin{lemma}\label{bridge}
For any bridge of a 1-nested $N,$ the split represented by that bridge is also represented by a bridge in $S_w(N).$
\end{lemma}
\begin{proof}
First we note that $S_w(N)$ does not subtract from the collection of bridges and cut-point nodes of $N.$ To see why: if $c$ is a circular order consistent with $N$, then $d_N$ is Kalmanson with respect to that circular order $c.$ Thus $c$ is also consistent with  $S_w(N),$ by Lemma~\ref{kal}.  Therefore, since the set of circular orders consistent with $N$ is determined by twisting around the splits associated to bridges or cut-point nodes of $N$, every bridge or cut-point node of $N$ must correspond to a bridge or cut-point node of $S_w(N),$ else some circular order would no longer be consistent. 

Our claim that the collection of bridges is not decreased by $S_w$ follows: if $e$ is a bridge in $N$ separating leaves $a$ and $b$ from leaves $c$ and $d$, then every path from $a$ to $c$ or $d$, and every path from $b$ to $c$ or $d$, must use $e.$ Let $x$ be the length of $e$, $l$ be the minimum distance from $a$ to the nearest endpoint of $e$, and $r$, $p$ and $q$ the minimum distances in $N$ from $b,c,d$ respectively to the nearest endpoint of $e.$ We know that $e$ corresponds to either a bridge or a cut-point node in $S_w(N),$ and the latter is equivalent to a bridge of zero length. We show that there is a bridge, of weight greater than or equal to $x.$ We assume a bridge of weight $x+\epsilon$ and show that $\epsilon \ge 0.$

In $s= S_w(N)$, with the bridge $e'$ of weight $x +\epsilon,$ the leaves $a,b,c,d$ have distances $l',r',p',q'$ to the ends of $e'.$ However, since $\mathbf{d}_s(a,b) = l' + r'$ and $\mathbf{d}_s(c,d) = p' + q',$ we have $l+r \ge l'+r'$ and $p+q \ge p'+q'.$ Also, $$\mathbf{d}_s(a,c) = l' + x +\epsilon +p' = l+x+p$$ and $$\mathbf{d}_s(b,d) = r' + x+ \epsilon +q' = r+x+q.$$  Adding and simplifying: $$l'+r'+p'+q' + 2\epsilon = l+r+p+q.$$ Using the inequalities, and subtracting, $$2\epsilon \ge 0.$$
Thus $\epsilon\ge 0$, and so the bridge cannot shrink, only possibly grow. A similar argument is constructed easily for the case where the bridge $e$ separates leaf $a$ from all other leaves, that is when $e$ is trivial.
\end{proof}

When we restrict to weighted circular split networks arising from weighted 1-nested networks, the codomain of $S_w$ is the outer-path circular split networks, and the distance vector is preserved by the map $L_w.$ Specifically we have: 
\begin{lemma}\label{lemw}
For any weighted 1-nested phylogenetic network $N$, if $s=S_w(N)$ then $s$ is outer-path and thus $\mathbf{d}_{L_w(s)} = \mathbf{d}_s.$
\end{lemma}
\begin{proof}
The distance $\mathbf{d}_s(i,j)$ is the sum of the weights of splits separating $i,j.$ It can also be seen as the sum of the weights of the edges in the split network on any shortest path from $i$ to $j,$ since every shortest path must use one edge representing each of the separating splits. Since each split is represented by a minimal cut (made up of parallel edges all the same weight) any shortest path from $i$ to $j$ has total weight no greater  than any path using edges on the exterior of the split network. If there is such a path on the exterior that is also minimal then the minimal weight of a path in $L_w(s)$ equals  $\mathbf{d}_s(i,j).$ Therefore we demonstrate the Lemma via the contrapositive. We show that if there is a shortest path partly through the interior of $s=S_w(N)$ that is strictly less than any on the exterior, then  $N$ cannot be 1-nested (nor 0-nested.) Note that for $n\le 5$ we can  inspect all the shapes of the circular split networks and see that any path between leaves which uses each split at most once is already visibly equal to a path on the exterior. (For a picture see \cite{dev-petti}.) For $n\ge 6$, if there is a path through the interior shorter than either two on the exterior, then there are 6 leaves that have the relationship in a (sub)-network  which we show in Figure~\ref{shortcut}, with variables representing positive lengths. The chord has the length of the interior short path in $s$, and the other edges the lengths of the exterior portions of $s.$ Thus we have that $g<f+x+h$, $f<g+h+x,$ $h<g+f+x$ and $x< f+g+h.$ Now for $s$ to equal $S_w(N)$ for $N$ a 1-nested network, the chord must disappear---at that point the remaining edges would have new lengths which preserve the existing distances between pairs of leaves. However, if this is possible then upon removing the chord the four leaves 1,...,4 will make a tree, and their distances will obey the additive conditions. Imposing additive conditions on the pairwise distances between the four nodes, regardless of whether the distance from 1 to 3 uses $x+h$ or $f+g$, and whether the distance from 2 to 4 uses $x+f$ or $g+h$, always forces one of the lengths $x,f,g,h$ to be zero, a contradiction. For instance, if $d_{1,3}+d_{2,4} \le d_{1,2}+d_{3,4}$ (picking two of the possible lengths and one of the conditions) then $a+x+h+c+b+g+h+d \le a+x+d+b+g+c$ which implies $2h\le 0.$ \begin{figure}
    \centering
    \includegraphics[width=4.5in]{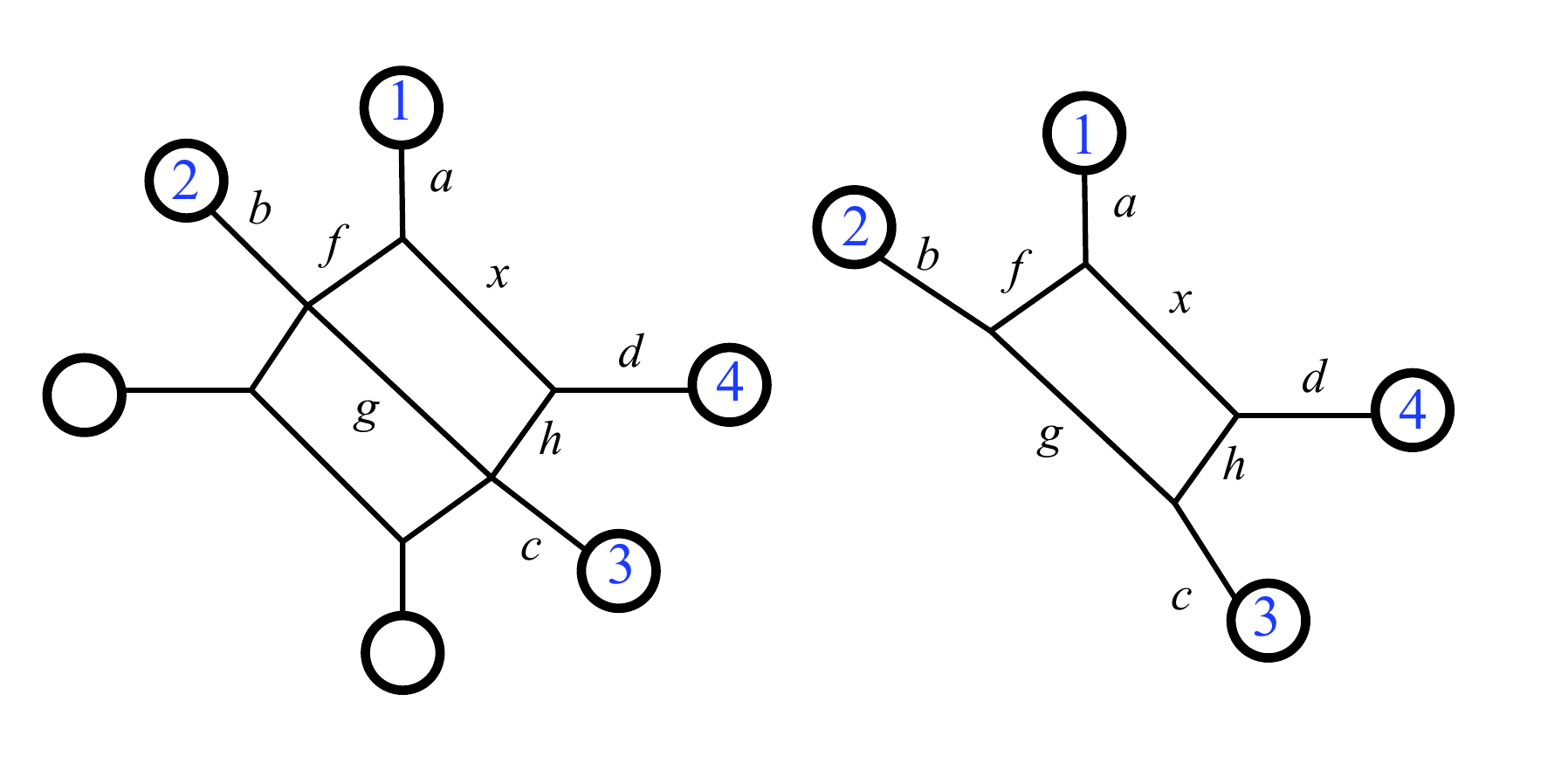}
    \caption{No shortcuts allowed, as described in proof of Lemma~\ref{lemw}.}    \label{shortcut}
\end{figure}\end{proof}

\subsection{Galois connection: weighted networks}
~

First we note that $L_w$ and $S_w$ are both monotone.   Since the only relations in our poset of weighted circular split networks are equalities, then $L_w$ is trivially monotone. Since comparable weighted phylogenetic networks have the same associated distance vector, then if $N\le N'$ we have $S_w(N) = S_w(N').$ When restricted to the weighted trees, $L_w$ and $S_w$ are both the identity function.

\begin{thm}\label{wgalois}
For any weighted outer-path circular split system $s$ and weighted 1-nested network $N,$ 

$$L_w(s) \le  N \text{ if and only if } s \le S_w(N) $$

That is, $L_w$ and $S_w$ form a Galois connection between weighted outer-path circular split systems and weighted 1-nested phylogenetic networks in which $L_w$ is the lower and $S_w$ the upper adjoint.  
\end{thm}
\begin{proof}
If for some outer-path network $s$ we are given $N$ with $L_w(s) \le N$ then the distance vectors are equal: $\mathbf{d}_{L_w(s)} = \mathbf{d}_N.$ Thus by the construction of $L_w$ on an outer-path network, we see $\mathbf{d}_s = \mathbf{d}_N.$ Therefore $S_w(N) = s,$ by unique representation of Kalmanson vectors.

Conversely, consider that for some $N$, we are given $s \le S_w(N).$  Then $s = S_w(N)$ as mentioned in Definition~\ref{defw}, and therefore $\mathbf{d}_{L_w(s)} = \mathbf{d}_N$ by Lemma~\ref{lemw}. Now we only need to show that the set of splits of $L_w(s)$ is contained in the set of splits of $N.$ We claim that any split $A|B$ of $S_w(N)$ is also a split of $N$. For the trivial splits we have Lemma~\ref{bridge}.  For $A|B$ nontrivial we have by Lemma~\ref{kal} that $A$ is contiguous in any circular order $c$ for which $\mathbf{d}$ is Kalmanson. Therefore $A$ is contiguous in $c$ for $c$ consistent with $N$. Therefore $A|B$ is in $N$ by Lemma~\ref{contig}, as claimed.
 Thus the splits of $s = S_w(N)$ are a subset of the splits of $\Sigma(N)$. Then by our earlier Galois connection from Theorem~\ref{unwgalois}, we have that the splits of $L(s)$ are a subset of the splits of $N$, and thus $L_w(s) \le N$, since $L(s)$ and $L_w(s)$ have the same set of splits. 
\end{proof}

\begin{figure}
    \centering
    \includegraphics[width=\textwidth]{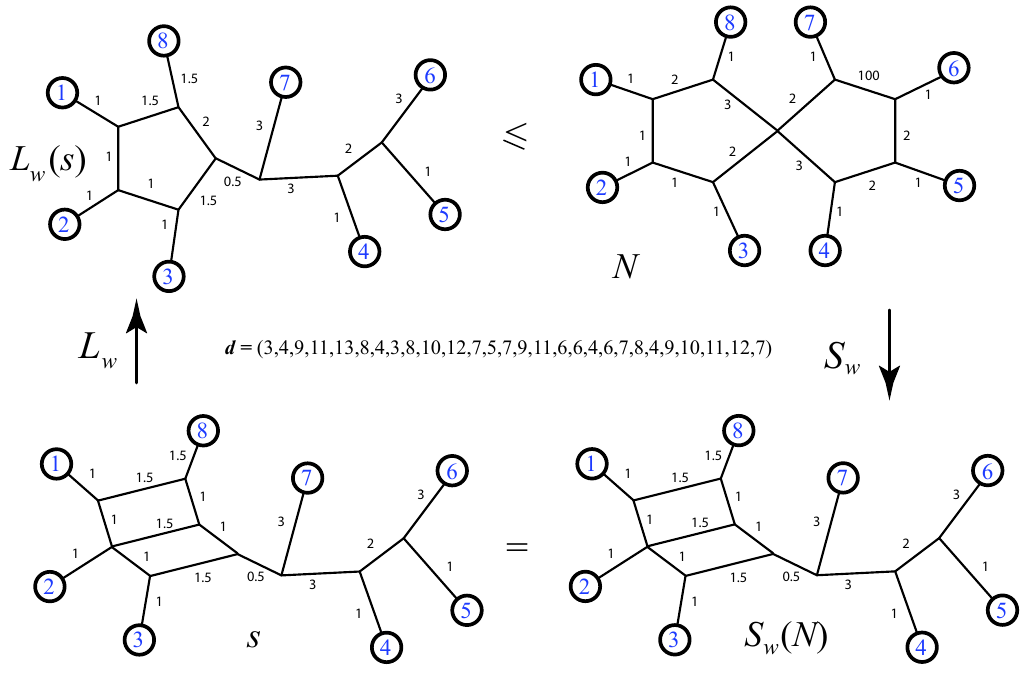}
    \caption{Example of the Galois connection for weighted networks. Here the central vector is Kalmanson, and is shared by all the networks: $\mathbf{d} = \mathbf{d}_N = \mathbf{d}_s.$}
    \label{fig:galois_w}
\end{figure}

\begin{thm}
 The Galois connection via $L_w$ and $S_w$ is a coreflection (but not a poset isomorphism).
\end{thm}
\begin{proof}
 We demonstrate that $S_w$ is surjective, onto the outer-path circular split networks.  We also point out that $S_w$ is not injective. 
First we show that for every weighted outer-path circular split network $s$, there exists a 1-nested weighted
network $N$, such that $S_w(N) = s.$ Given $s$ outer-path, let $N= L_w(s).$  Then $\mathbf{d}_N = \mathbf{d}_s$ by construction. Thus $S_w(N) = s$,
since neighbor-net always gives the unique weighted split network for any Kalmanson metric $\mathbf{d}.$ Next, $S_w$ is not injective. To see this, consider the two weighted networks in the upper inequality pictured in Figure~\ref{fig:galois_w}. Their respective sets of splits are unequal, but they give rise to the same distance vector $\mathbf{d}.$ Thus they have the same image under $S_w.$   
\end{proof}

As corollaries (by standard Galois theory) it is implied that $L_w$ is injective, from the outer-path circular split networks, but not surjective. For an example of the non-surjectivity of $L_w$ we can observe that a cycle of length 4 with one side of weight larger than all the others will never arise as the image of $L_w,$ since cycles of length 4 would be preserved as such, but with pairs of matching weights. Note that this is exactly the reverse of the situation for the unweighted version of the map, $L.$ In the weighted version, $L_w$ is one-to-one and $S_w$ provides an inverse function when restricted to the range of $L_w$ (in turn, restricted to the domain of outer-path circular split networks.) Thus $S_w\circ L_w$ is the identity on outer-path circular split networks, but $L_w \circ S_w$ is decreasing. Figure~\ref{fig:galois_w} exhibits the latter case directly.

\section{Implications for Polytopes}\label{sec_poly}
Recently in \cite{durell-forcey} we described for each $n$ a sequence of polytopes that interpolate between the well-known Symmetric Travelling Salesman Polytope (STSP($n$)) and the Balanced Minimum Evolution Polytope (BME($n$)). The new polytopes are called the level-1 network polytopes BME($n,k$) for $0\le k \le n-3$. Each is of dimension ${n \choose 2} - n.$ After scaling, all of their vertices are located at barycenters of the faces of STSP($n$), and each  BME($n,k$) is nested inside of BME($n,j$) for $j\le k.$  In this nested polytope picture, the largest is BME($n,0$) which is (a scaled version of) STSP($n$) and the smallest is BME($n,n-3$) = BME($n$). Here we review some basic definitions and results, and then discuss new insights.

\begin{definition}\label{e:bmenkvert} For a binary, 1-nested phylogenetic network $N$, the vector ${\mathbf x}(N)$ is defined to have lexicographically ordered components ${x}_{ij}(N)$ for each unordered pair of distinct leaves $i,j \in [n]$ as follows, where $C[n]$ is the cyclic orders of $[n]$.

$$ {x}_{ij}(N) = \begin{cases} 2^{k-b_{ij}} & \text{if there exists $c\in C[n]$ consistent with $N$; with $i,j$  adjacent in $c$,}\\ 0 & \text{otherwise.} \end{cases}
$$

where $k$ is the number of bridges in $N$ and $b_{ij}$ is the number of bridges crossed on any path from $i$ to $j$.

The  convex  hull  of  all  the   ${\mathbf x}(N)$ such that binary $N$ has $k$ nontrivial bridges is the level-1 network polytope BME($n,k$). As shown in \cite{durell-forcey}, the vertices of BME($n, k$) are precisely the vectors ${\mathbf x}(N)$ for $N$ binary with $n$ leaves and $k$ nontrivial bridges.  \end{definition}
 
Also as shown in \cite{durell-forcey}, an equivalent definition of the vector $\mathbf{x}(N)$ is the vector sum of the vertices of the STSP($n$) which correspond to cyclic orders consistent with $N$. The vertices of STSP($n$) are the incidence vectors $\mathbf{x}(c)$ for each cyclic order $c$ of $n$, where the $i,j$ component is 1 for $i$ and $j$ adjacent in the order $c$, 0 otherwise. This alternative definition may be applied to any 1-nested phylogenetic network, not just the binary ones.
\begin{definition}
For a 1-nested phylognetic network $N$, the vector $\mathbf{x}(N) = \sum \mathbf{x}(c)$ where the sum is over all cyclic orders of $[n]$ consistent with $N.$ 
\end{definition}
Note that for phylogenetic trees  $t$ (with nodes of any degree), this definition $\mathbf{x}(t)$ agrees with the definition of the coefficient $n_t$ in \cite{Steel}, in the proof of Theorem 4.2 of that paper.

 A large body of knowledge exists about the facets of BME($n,k$), especially if we include the special cases of $k=0$ and $k=n-3$. For $k=0$ the vertices are cyclic orders, and the polytope BME($n,0)$ is the Symmetric Travelling Salesman polytope. For $k=n-3$ the vertices are phylogenetic trees and the polytope BME($n,n-3$) = BME($n$).   The facets of STSP($n$) are well studied, from \cite{dantzig}, to \cite{grpa1} and \cite{grpa2}, with a nice survey in \cite{rodin}. The facets of BME($n$) are first described in \cite{forcey2015facets} and  \cite{splito}. A class of facets shared by all the polytopes BME($n,k$) are the split facets: each corresponds to a nontrivial split of $[n],$ as shown in \cite{durell-forcey}. (For $k=n-3$ the split must have parts larger than 3, it is conjectured that this is not necessary for $0\le k < n-3$).

Two polytopes are \emph{nested} when one is contained in the other, with all vertices of the smaller on faces of the larger. In \cite{durell-forcey} it is shown that for any $n$ the scaled polytopes ($2^{n-3-k}$)BME($n,k$) are sequentially nested, from $k=0$, the largest, to $k=n-3$, the smallest. Each vertex of a smaller scaled polytope is at the barycenter of a face of BME($n,0$).
Figure~\ref{splitfacet} shows a facet of BME(5,0) which corresponds to the split $\{1, 2\}|\{3, 4, 5\}.$
\begin{figure}
    \centering
    \includegraphics{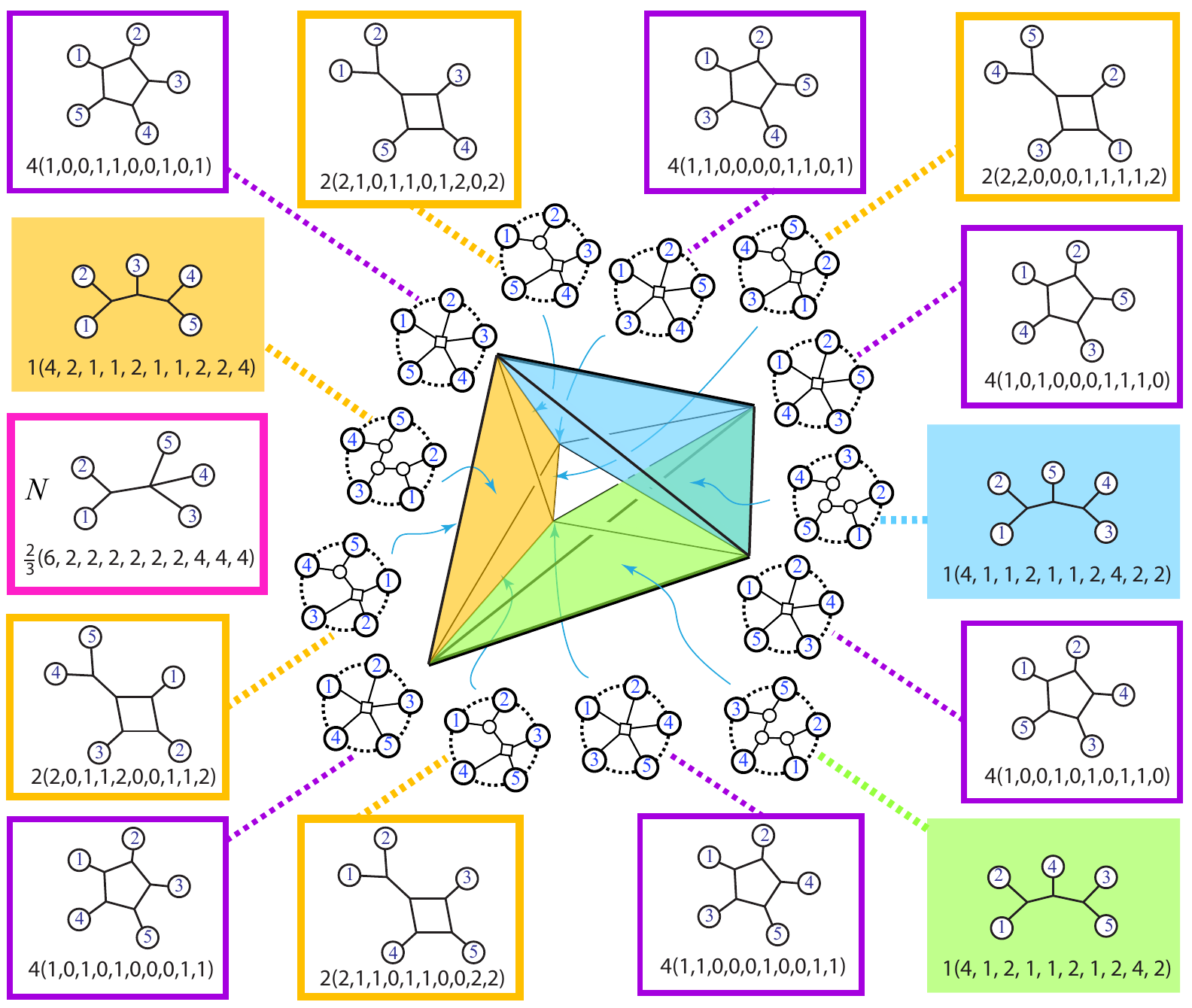}
    \caption{The scaled split facet $F_0(N)$ of BME(5,0) = STSP(5). This 4D facet corresponds to the split $\{1, 2\}|\{3, 4, 5\}$ (pictured as the tree $N$, center-left), and is also known as a subtour elimination facet. Three of its tetrahedral subfaces correspond to networks, and are shaded. Vertices and some faces are labeled with both networks and PC-trees. A (scaled version of) the vector $\mathbf{x}$  is shown  beneath each network: the barycenter of the face represented by that network.}
    \label{splitfacet}
\end{figure}

In \cite{durell-forcey}, Theorems 8 and 9, we show that for $s$ a weighted circular split network with $n$ leaves and $k$ bridges such that $L(\overline{s})$ is binary, the dot product $\mathbf{x}(N)\cdot \mathbf{d}_s$ is minimized uniquely over BME($n,k$) at the vertex $\mathbf{x}(L(\overline{s}))$. Furthermore, for a given 1-nested network $N$ we can often find it as a face in multiple polytopes. In fact, from Theorem 11 of \cite{durell-forcey}, we have:

\begin{theorem}\label{t:faces}
Every $n$ leaved 1-nested unweighted network $N$ with $m$ bridges corresponds to a face $F_k(N)$
of each polytope BME($n,k$)  for $0\le k \le m$. \vspace{.12in} 

That face has vertices  $\mathbf{x}(N')$ for all the  binary 1-nested $k$-bridge networks $N'$ such that $N \le N'.$  \end{theorem}

The implication of the theorem just recounted is that the poset of 1-nested networks (up to equivalence) is found mirrored in the the face posets of the BME($n,k$) polytopes. If $N\le N'$ then $F_k(N')\subseteq F_k(N).$ This follows easily since the set of vertices of $F_k(N)$ will contain the set of vertices of $F_k(N').$ Note that the set of binary 1-nested networks refining the splits displayed by a given 1-nested network depends only on the split system displayed. Thus we can restate the result in terms of $PC$-trees as follows:

\begin{corollary}
The $PC$-trees on $[n]$, ordered by reverse containment of splits, are isomorphic to a sub-poset of the face poset of BME($n,0$) = STSP($n$). Subposets of this poset are also found within faces of each BME($n,k$) for $0\le k \le n-3$.
\end{corollary}
\begin{proof}
The polytope BME($n,k$) has a subposet of faces, ordered by inclusion,  with maximal elements  represented by the single splits on $[n].$ Subfaces of a face are found by adding splits to the network, but we claim that the resulting split systems must be represented by  $PC$-trees. That is true since the faces are represented by (equivalence classes of) 1-nested networks, in light of Theorem~\ref{t:faces} and Theorem~\ref{pctree}. This identification of subfaces continues until one reaches the binary 1-nested networks with $k$ nontrivial bridges. Those have the maximum number of splits and are the vertices of the polytope BME($n,k$). Note that binary 1-nested  networks are in bijection with $PC$-trees for which the all nodes are cyclic (class $C$) except for the nodes of degree 3 (which are permutable, class $P$.)  Thus BME($n,0$) = STSP($n$) has a face for every  $PC$-tree on $[n]$, since the vertices are those networks with no trivial bridges, which are in bijection with the cyclic orders of $[n].$
\end{proof}

 We can use the Galois connections for both weighted and unweighted networks to more fully describe how the Kalmanson metrics $\mathbf{d}$ relate to the BME($n,k$) polytopes. It turns out that the unique split network $S_w(N)$ associated to a weighted 1-nested phylogenetic network $N$ has a set of splits which are all displayed by the binary networks at which the dot product with $\mathbf{d}_N$ is minimized.  
\begin{figure}
    \centering
    \includegraphics[width=4in]{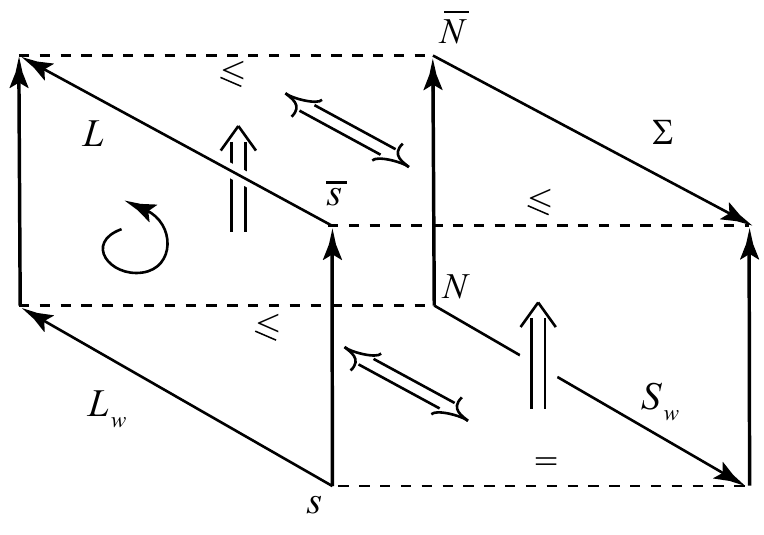}
    \caption{Vertical maps shown here are forgetting the weights. The two Galois connections are shown as biconditionals (double arrows). The quadrilateral at the left commutes, on the right does not. The implication shown by the front rectangle is the inequality in the proof of Theorem~\ref{newth}.}
    \label{fig:galois_b}
\end{figure}

\begin{thm}\label{newth}
 Given any weighted 1-nested phylogenetic network ${N}$ with $n$ leaves, the product $\mathbf{x}(\hat{N})\cdot \mathbf{d}_{{N}}$ is minimized over BME($n,k$)  precisely for the unweighted binary networks $\hat{N}$ with $k$ bridges such that $\overline{S_w({N})} \le \Sigma(\hat{N})$.\\
\end{thm}

\begin{proof}
 We know from Theorems 8, 9, and 11 of \cite{durell-forcey} that the dot product is minimized precisely for binary networks $\hat{N}$ with $k$ bridges such that $\overline{N} \le \hat{N}.$ (These theorems are repeated here in Definition~\ref{e:bmenkvert} and Theorem~\ref{t:faces}.) Thus  the dot product is minimized if and only if $\Sigma(\overline{N}) \le \Sigma(\hat{N}).$
 We also claim that for any weighted 1-nested phylogentic network $N,$ the following inequality holds:  $$\overline{S_w(N)}\le \Sigma(\overline{N}).$$ The claim follows from the theorems of this paper, as illustrated in Figure~\ref{fig:galois_b}. Let $s= S_w(N).$ First, $s = S_w(N)$ if and only if $L_w(s) \le N$, by Theorem~\ref{wgalois}. Next $L_w(s) \le N$ implies that $\overline{L_w(s)}\le \overline{N}$ by Lemma~\ref{overl}, which implies that $L(\overline{s})\le \overline{N}$ by Lemma~\ref{comm}. The latter inequality holds, by Theorem~\ref{unwgalois}, if and only if $\overline{s} \le \Sigma(\overline{N}).$ Thus we have $\overline{S_w(N)} \le \Sigma(\overline{N}) \le \Sigma(\hat{N}).$ 
\end{proof}


Notice that Theorem~\ref{newth} does not mention the number of bridges of $N.$ Thus the number of bridges of the networks $\hat{N}$ where the dot product is minimized can vary. This is seen in Examples~\ref{ex1} and ~\ref{ex2}. Example~\ref{ex3} shows some networks $\hat{N}$ which fail to meet the criteria and thus exemplify the strictness of the minimization. In terms of the face structure of the polytopes, we can say that the face associated to any unweighted 1-nested network $\overline{N}$ is a subface of any face associated to the exterior 1-nested network of the unique split network corresponding to a weighted version $N.$

\begin{thm}\label{subth}
 Given any weighted 1-nested phylogenetic network $N$, with $m\ge k$ bridges, there is a face $F_k(\overline{N})$ of BME$(n,k)$  which is a subface of $F_k(L(\overline{S_w(N)}))$.
\end{thm}
\begin{proof}
 By definition, $L(\overline{S_w(N)}) = \overline{L_w(S_w(N))}.$ By the weighted Galois connection of Theorem~\ref{wgalois}, $L_w(S_w(N))\le N.$ Thus the unweighted versions have the same relationship: $\overline{L_w(S_w(N))}\le \overline{N},$ and so the face corresponding to $\overline{N}$ is a subface of the face corresponding to $\overline{L_w(S_w(N))}.$ 
 \end{proof}
 
The statement of Theorem~\ref{subth} is equivalent to saying that the vertices in the face corresponding to $L(\overline{S_w(N)})$ contain as subset the vertices in the face corresponding to $\overline{N}$. Thus we have the following: 
 
\begin{corollary}
 Given any Kalmanson metric $\mathbf{d}$ on $[n]$ , the product $\mathbf{x}(\hat{N})\cdot \mathbf{d}$ is minimized simultaneously for the binary networks $\hat{N}$ with $k$ bridges such that $\overline{\mathcal{N}(\mathbf{d})} \le \Sigma(\hat{N})$,  where $\mathcal{N}(\mathbf{d})$ denotes the output of the neighbor-net algorithm.\\
\end{corollary}

Thus $L_w(\mathcal{N}(\mathbf{d}))$ is a good candidate for the best fit of any 1-nested phylogenetic network to a given $\mathbf{d}.$  Since $\Sigma$  and $S_w$ preserve bridges, we have:

\begin{corollary}
Given a Kalmanson vector $\mathbf{d}$ on $[n],$ the number $k$ of nontrivial bridges in $\mathcal{N}(\mathbf{d})$ is also the smallest value of $k$ such that BME($n,k$) has a unique vertex $\mathbf{x}(N)$ at which the dot product with $\mathbf{d}$ is minimized over that polytope. 
\end{corollary}

\begin{example}\label{ex1}
Consider the weighted network $N$ in Figure~\ref{fig:galois_w}. Figure~\ref{fig:bigo_k1} shows two binary networks $\hat{N}, \hat{N}'$ with $k=1$ bridge such that $\overline{S_w(N)} \le \Sigma(\hat{N})$ and $\overline{S_w(N)} \le \Sigma(\hat{N}').$ The weight $W(S_w(N)) = 21.5$, and (both) dot products $\mathbf{x}(\hat{N})\cdot \mathbf{d}_N = 4(21.5) = 86.$
\begin{figure}
    \centering
    \includegraphics[width=\textwidth]{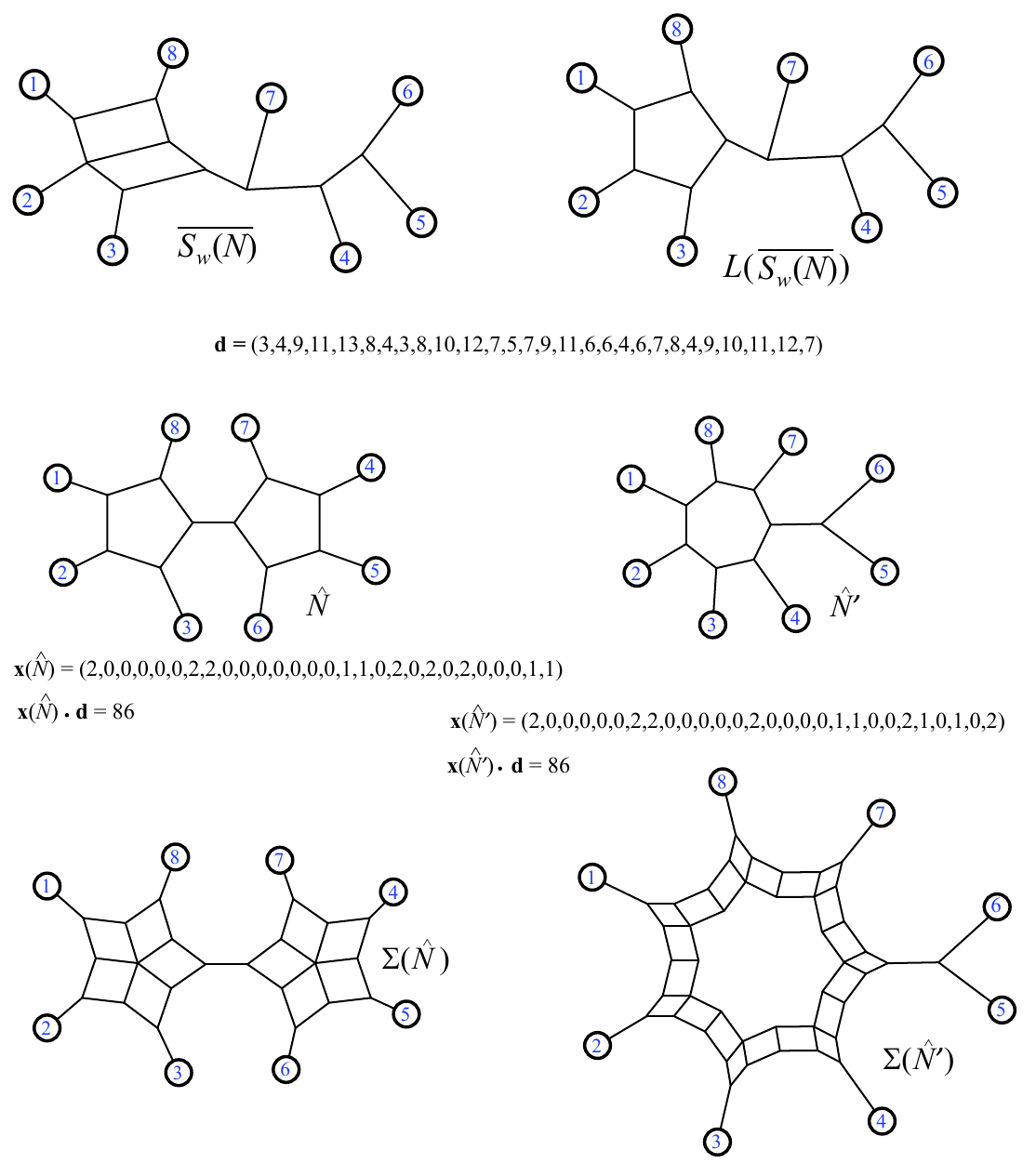}
   \caption{$\hat{N}$ and $\hat{N}'$ are two minimizing vertices as predicted by Theorem~\ref{newth}, for Example~\ref{ex1}.}
    \label{fig:bigo_k1}
\end{figure}
\end{example}

\begin{example}\label{ex2}
Again consider the weighted network $N$ in Figure~\ref{fig:galois_w}. Figure~\ref{fig:bigo_k2} shows two binary networks $\hat{N}_1, \hat{N}_2$ with $k=2$ bridges such that $\overline{S_w(N)} \le \Sigma(\hat{N}_1)$ and $\overline{S_w(N)} \le \Sigma(\hat{N}_2).$ The weight $W(S_w(N)) = 21.5$, and (both) dot products $\mathbf{x}(\hat{N})\cdot \mathbf{d}_N = 8(21.5) = 172.$
\begin{figure}
    \centering
    \includegraphics[width=\textwidth]{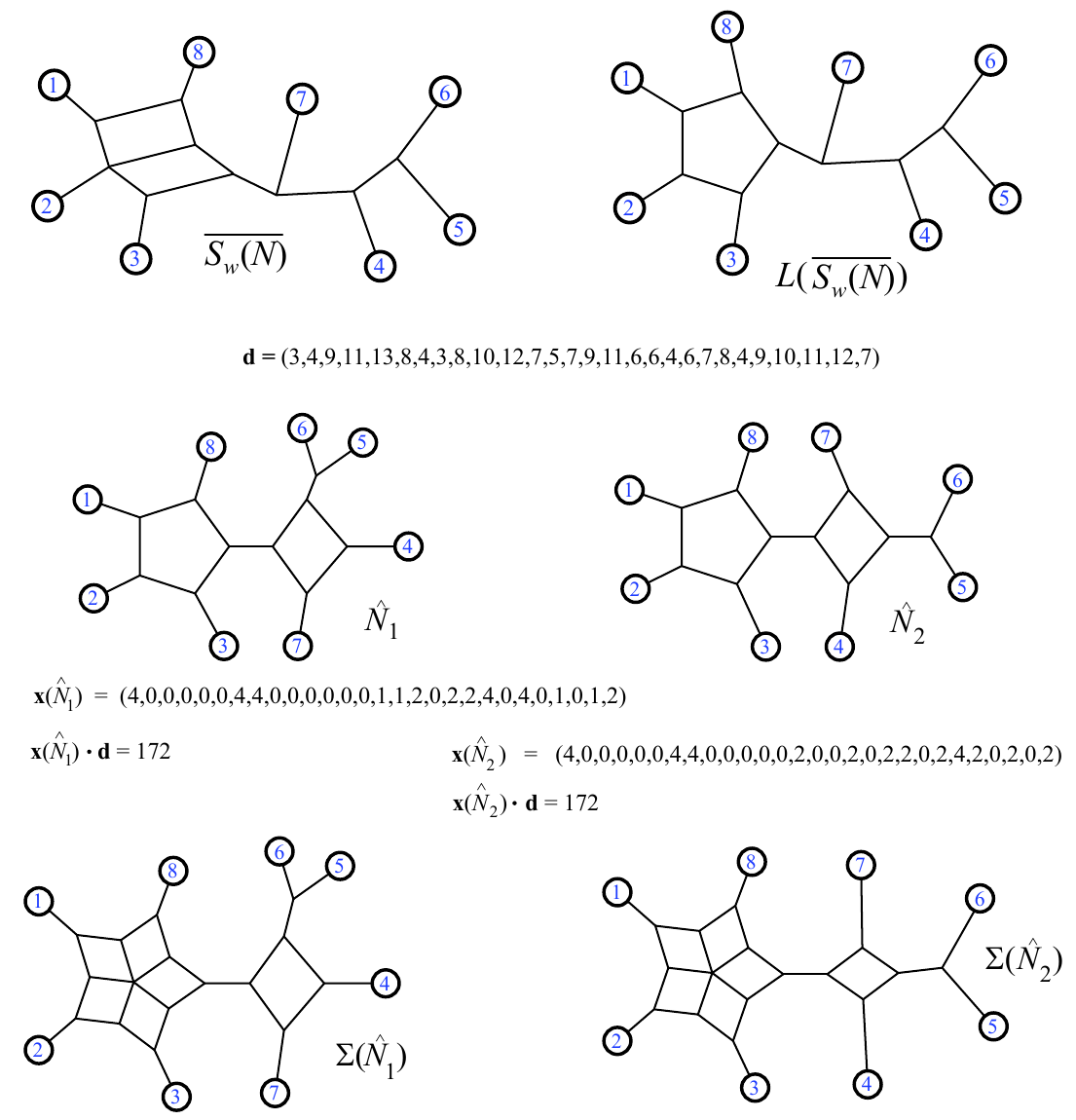}
    \caption{$\hat{N}_1$ and $\hat{N}_2$ are two minimizing vertices as predicted by Theorem~\ref{newth}, for Example~\ref{ex2}.}
    \label{fig:bigo_k2}
\end{figure}

\end{example}

\begin{example}\label{ex3}
Again consider the weighted network $N$ in Figure~\ref{fig:galois_w}. Figure~\ref{fig:bigo_kb} shows two binary networks $\hat{N}, \hat{N'}$ which are missing splits that are displayed by $S_w(N)$. On the left the network does not display the split $\{5,6\}|\{1,2,3,4,7,8\}$.  On the right the network does not display the split $\{1,8\}|\{2,3,4,5,6,7\}$. Their respective dot products are larger than for the minimizing vertices shown in the previous two examples. 
\begin{figure}
    \centering
   \includegraphics[width=\textwidth]{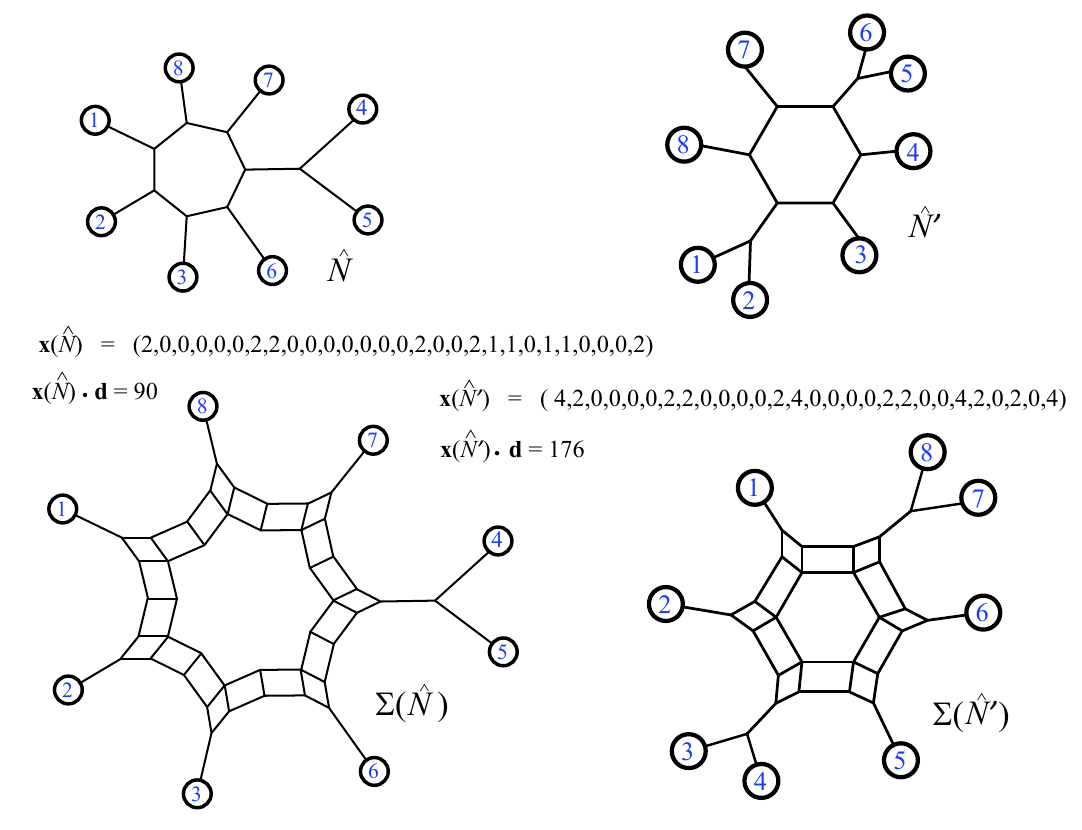}
    \caption{$\hat{N}$ and $\hat{N}'$ are two non-minimizing vertices as predicted by Theorem~\ref{newth}, for Example~\ref{ex3}.}
    \label{fig:bigo_kb}
\end{figure}

\end{example}

\begin{example}\label{ex4}
In Figure~\ref{bito} we show some of the 1-nested phylogenetic networks from the previous examples, with some other networks for context. They are arranged in the containment order of the faces of BME(8,2) from top to bottom. Figure~\ref{bitopc} show the same portion of the poset, but with the representative $PC$-trees. 
\begin{figure}
    \centering
    \includegraphics[width=\textwidth]{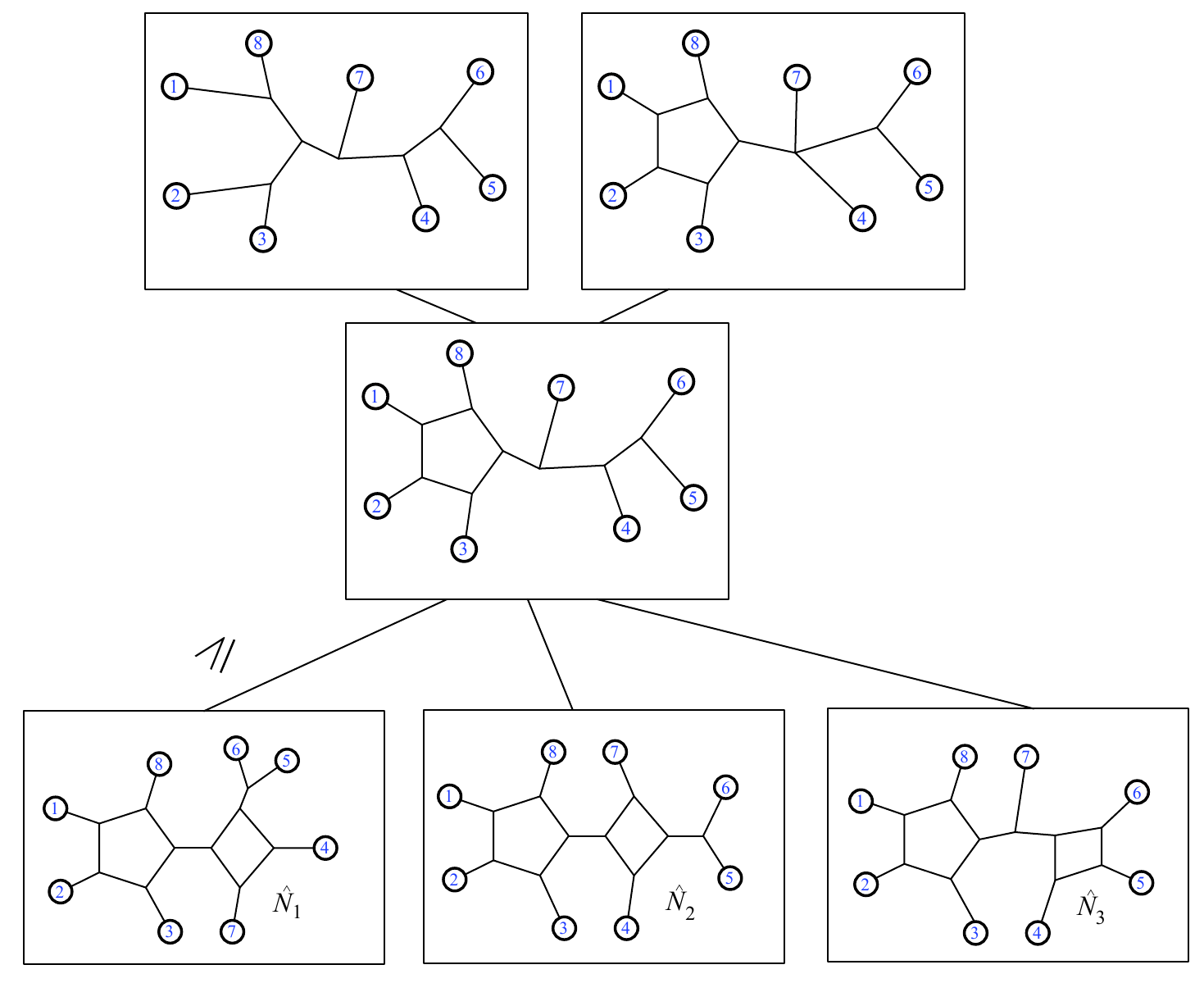}
    \caption{ Faces higher on the page contain lower faces  they are connected to by lines,  in BME($8,k$) for $k \le 2$. (Containment of sets of displayed splits is in the other direction.) In the center is a vertex of BME($8,3$). However it is also a face of BME(8,2),  BME(8,1) and BME (8,0). At the bottom are binary 1-nested networks: vertices of BME(8,2), thus displaying as many splits as possible with two bridges. Note that the relations shown here may not be covering relations in the polytopes. }
    
    \label{bito}
\end{figure}
 \end{example}

\begin{figure}
    \centering
    \includegraphics[width=\textwidth]{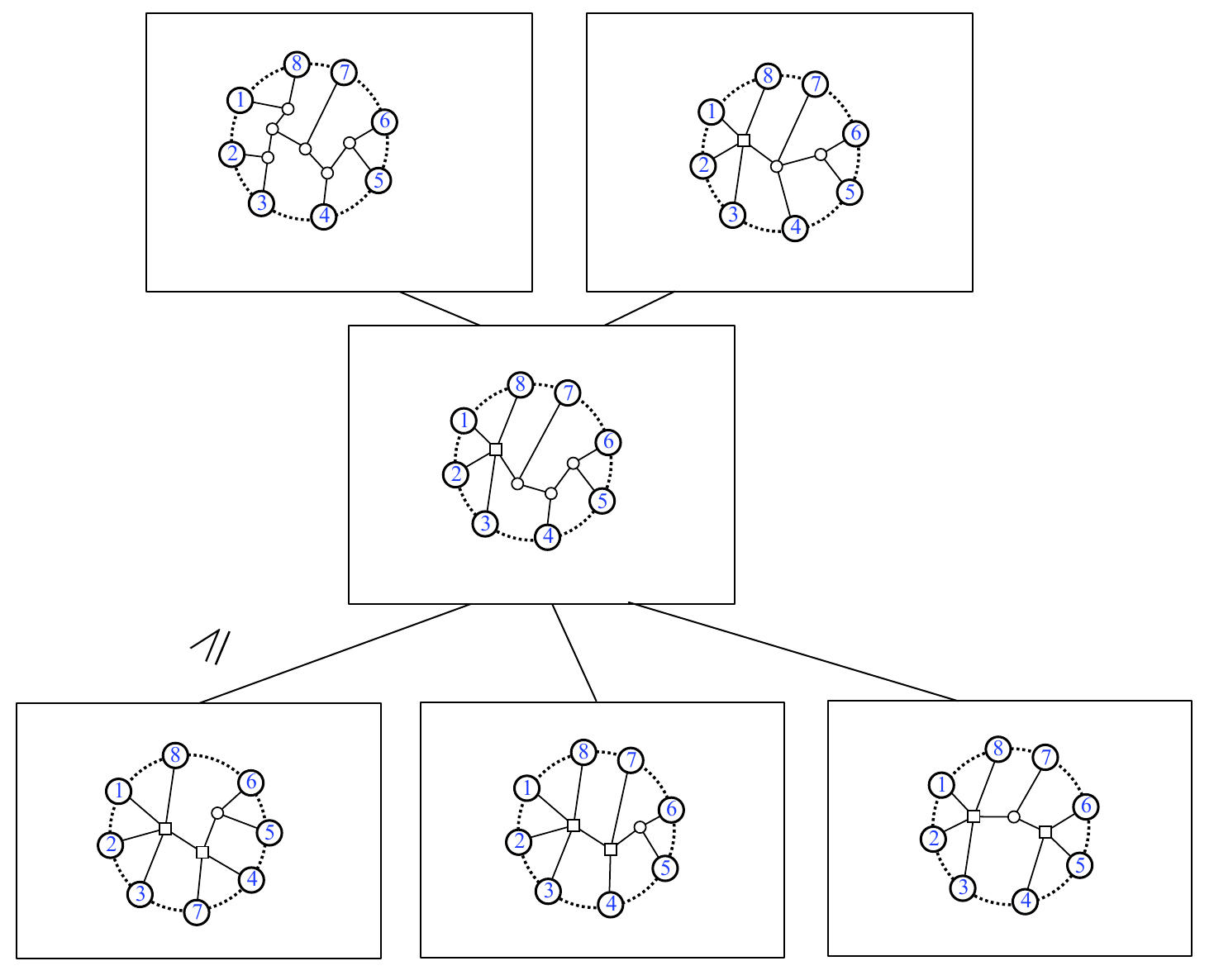}
    \caption{These $PC$-trees are the ones in correspondence with the 1-nested networks in the same respective positions in Figure~\ref{bito}.}
    
    \label{bitopc}
\end{figure}

\bibliographystyle{amsplain}
\bibliography{galois}{}

\end{document}